\documentclass[onetabnum,onefignum,nohypdvips,final]{siamart171218}
\usepackage{amssymb,amsmath,epsfig,verbatim}
\ifpdf
  \DeclareGraphicsExtensions{.eps,.pdf,.png,.jpg}
\else
  \DeclareGraphicsExtensions{.eps,.ps}
\fi

\newtheorem{ass}[theorem]{Assumption}

\newcommand{\bRgeq}{{\mathbb R}_{\geq 0}}

\newcommand{\bR}{{\mathbb R}}

\newcommand{\bZ}{\mathbb{Z}}
\newcommand{\drho}{\;{\rm d}\rho}
\newcommand{\dzeta}{\;{\rm d}\zeta}

\newcommand{\id}{{\rm id}}
\newcommand{\dd}[1]{\frac{\rm d}{{\rm d}#1}}
\newcommand{\ddt}{\dd{t}}
\newcommand{\ek}{e}
\newcommand{\ttau}{\Delta t}
\def\epsilon{\varepsilon}

\begin{document}
\title{
Error analysis for a finite difference scheme \\ for 
axisymmetric mean curvature flow \\ of genus-0 surfaces%
}
\author{Klaus Deckelnick\footnotemark[2]\ \and 
        Robert N\"urnberg\footnotemark[3]}

\renewcommand{\thefootnote}{\fnsymbol{footnote}}
\footnotetext[2]{Institut f\"ur Analysis und Numerik,
Otto-von-Guericke-Universit\"at Magdeburg, 39106 Magdeburg, Germany}
\footnotetext[3]{Department of Mathematics, University of Trento, Trento,
Italy}

\date{}

\maketitle

\begin{abstract}
We consider a finite difference approximation of mean curvature flow
for axisymmetric surfaces of genus zero. A careful treatment of the 
degeneracy at the axis of rotation for the one dimensional partial differential
equation for a parameterization of the generating curve allows us to
prove error bounds with respect to discrete $L^2$-- 
and $H^1$--norms for a fully discrete approximation. 
The theoretical results are confirmed with the help of numerical
convergence experiments.
We also present numerical simulations for
some genus-0 surfaces, including for a non-embedded self-shrinker
for mean curvature flow.
\end{abstract} 


\renewcommand{\thefootnote}{\arabic{footnote}}

\pagestyle{myheadings}
\thispagestyle{plain}
\markboth{K. DECKELNICK AND R. N\"URNBERG}
{ERROR ANALYSIS FOR AXISYMMETRIC MEAN CURVATURE FLOW}

\setcounter{equation}{0}
\section{Introduction} \label{sec:intro}

Consider a family of surfaces $(\mathcal S(t))_{t \in [0,T]} \subset \bR^3$ 
evolving by mean curvature flow, i.e.\
\begin{equation} \label{eq:mcfS}
\mathcal{V}_{\mathcal{S}} = k_m\qquad\text{ on } \mathcal{S}(t).
\end{equation}
Here, $\mathcal{V}_{\mathcal{S}}$ denotes the normal velocity of 
$\mathcal{S}(t)$ in the direction of the normal $\vec\nu_{\mathcal{S}(t)}$,
and $k_m$ is the mean curvature of $\mathcal{S}(t)$, i.e.\ the sum of its
principal curvatures. As the $L^2$--gradient flow for the area functional, \eqref{eq:mcfS} is one
of the most important geometric evolution equations with applications in materials science and image
processing. We refer the reader to \cite{Ecker04,Mantegazza11} for an 
introduction and important results of mean curvature flow.

In this paper we are concerned with the numerical approximation of solutions of \eqref{eq:mcfS} using
a parametric approach. 
If $\vec X: \mathcal M \times [0,T) \to \bR^3$ is a family of 
embeddings such that $\mathcal S(t) = \vec X(\mathcal M,t)$, 
then \eqref{eq:mcfS} is satisfied if $\vec X_t \circ \vec X^{-1} 
= k_m \vec\nu_{\mathcal{S}(t)}$ on $\mathcal{S}(t)$.
Making use of the fact that the mean curvature vector 
$k_m \vec\nu_{\mathcal{S}(t)}$ can be written as 
$\Delta_{\mathcal S(t)} \vec\id$, 
where $\Delta_{\mathcal S(t)}$ denotes the
Laplace--Beltrami operator on $\mathcal{S}(t)$,
Dziuk \cite{Dziuk91}
suggested a finite element method in order to approximate solutions of 
\eqref{eq:mcfS}. 
While this approach has been widely
used in the following years, the numerical analysis of the method remained 
open. Only recently, Kov{\'a}cs, Li and Lubich \cite{KovacsLL19}
obtained error estimates
for a parametric approach that uses not only the position $\vec X$, but also the mean curvature and the normal as variables.
Both approaches are based on
evolution equations in which the velocity vector points purely in normal direction, which may lead to degenerate meshes at the discrete level. 
A way to tackle this 
issue is to introduce a suitable additional tangential motion in such a way, 
that mesh points are better distributed on the approximate surface. 
Corresponding schemes have been
suggested by Barrett, Garcke and N\"urnberg \cite{gflows3d}, 
as well as by Elliott and Fritz \cite{ElliottF17}, using DeTurck's trick. 
For the approach from \cite{ElliottF17}, 
error bounds for a finite difference scheme in the case of surfaces of 
torus type have recently been obtained in \cite{Mierswa20}.
For more details on the numerical approximation of geometric evolution
equations we refer to the review articles \cite{DeckelnickDE05,bgnreview}. 

In what follows, we are interested in the case that the evolving surfaces 
are axisymmetric with respect to the $x_2$-axis, i.e.\ we assume that there 
exists a mapping
$\vec x(\cdot,t) : [0,1] \to \bRgeq\times \bR$ such that
\begin{equation*} 
\mathcal S(t) = \left\{ \left( \vec x(\rho,t) \cdot \vec e_1 \cos \theta, 
\vec x(\rho,t) \cdot \vec e_2, \vec x(\rho,t) \cdot \vec e_1 \sin \theta 
\right)^T : \rho \in [0,1], \theta \in [0, 2 \pi] \right\}.
\end{equation*}
As shown in \cite{aximcf,schemeD}, the law (\ref{eq:mcfS}) translates into the following evolution equation for the curves $(\Gamma(t))_{t \in [0,T]}$ 
parameterised by $\vec x(\cdot,t)$:
\begin{equation} \label{eq:xtflow}
\vec x_t\cdot\vec\nu  = 
\varkappa - \frac{\vec\nu\cdot\vec\ek_1}{\vec x\cdot\vec\ek_1} ,
\end{equation}
where $\vec\nu$ is a unit normal to $\Gamma(t)$ and
$\varkappa = \vec\varkappa\cdot\vec\nu$ denotes curvature, with
$\vec\varkappa = 
\frac1{|\vec x_\rho|}( \frac{\vec x_\rho}{|\vec x_\rho|} )_\rho$ 
the curvature vector. 
We note that without the last term on the right hand side of \eqref{eq:xtflow},
the problem collapses to curve shortening flow,
\begin{equation} \label{eq:csf}
\vec x_t \cdot \vec\nu = \varkappa,
\end{equation}
which is the analogue of \eqref{eq:mcfS} for curves. 
Since the relations \eqref{eq:xtflow} and \eqref{eq:csf} 
only prescribe the normal velocity, there is
a certain freedom in choosing the tangential part of the velocity vector.
Setting the tangential velocity to zero for \eqref{eq:csf} leads to the
formulation 
$\vec x_t = \frac1{|\vec x_\rho|}( \frac{\vec x_\rho}{|\vec x_\rho|} )_\rho$,
and optimal error bounds for a semidiscrete continuous-in-time finite element
approximation of it have been obtained by Dziuk \cite{Dziuk94}.
On the other hand, an application of DeTurck's trick gives rise to the
formulation $\vec x_t = \frac{\vec x_{\rho \rho}}{|\vec x_\rho|^2}$ for
classical curve shortening flow. An error analysis for a corresponding 
semidiscrete finite element scheme has been first presented in 
\cite{DeckelnickD95},
and this was later extended in \cite{ElliottF17} to the family of problems 
$\alpha\vec x_t + (1-\alpha) (\vec x_t \cdot \vec\nu) \vec\nu 
= \frac{\vec x_{\rho \rho}}{|\vec x_\rho|^2}$, $\alpha \in (0,1]$. 
Inspired by the ideas in \cite{DeckelnickD95}, the present authors 
in \cite{schemeD} applied 
DeTurck's trick to the flow \eqref{eq:xtflow} to obtain the system
\begin{equation} \label{eq:newsystem} 
\vec x_t = \frac{\vec x_{\rho \rho}}{|\vec x_\rho|^2} - \frac{\vec\nu\cdot\vec\ek_1}{\vec x\cdot\vec\ek_1} \vec\nu,
\end{equation}
cf.\ \cite[(1.7)]{schemeD}.
Note that \eqref{eq:newsystem} is strictly parabolic and that
a solution of \eqref{eq:newsystem} satisfies (\ref{eq:xtflow}). 
The difference to curve shortening flow consists in the presence of the term 
$\frac{\vec\nu\cdot\vec\ek_1}{\vec x\cdot\vec\ek_1}$, which is the principal 
curvature related to the parallels of $\mathcal S(t)$. 
It is possible to rewrite \eqref{eq:newsystem} in the following divergence form
\begin{equation} \label{eq:newsystemvar}
\vec x \cdot \vec e_1 \, | \vec x_\rho |^2 \vec x_t =
\bigl( ( \vec x \cdot \vec e_1) \vec x_\rho \bigr)_\rho 
- | \vec x_\rho |^2 \vec e_1,
\end{equation}
giving rise to a natural variational formulation. On the basis of this weak 
formulation, a semi-implicit scheme using piecewise linear finite elements in 
space and a backward Euler method in time was suggested by the authors
in \cite{schemeD}. 
In particular, in \cite[Theorem~2.2]{schemeD} optimal error bounds both in 
$H^1$ and $L^2$ 
are obtained in the case of genus-1 surfaces. While the numerical method still 
performs well also for genus-0 surfaces, it is however not possible to apply 
the employed analysis to genus-0 surfaces.
The reason for the additional difficulties 
in the genus-0 case comes from the different properties of the curves 
$\Gamma(t)$: for genus-1 surfaces, $\Gamma(t)$ is a closed curve satisfying 
$\vec x \cdot \vec e_1>0$ on $[0,1]$ so that this term is bounded strictly
from below on compact time intervals, thus simplifying the analysis. 
In contrast, a description of a genus-0 surface in our setting requires 
$\Gamma(t)$ to be open with its endpoints lying on the $x_2$-axis,
which means that $\vec x \cdot \vec e_1=0$ at the endpoints of the interval 
$[0,1]$.
Furthermore, in order to guarantee smoothness of the surface $\mathcal S(t)$, 
the curve $\Gamma(t)$ has to meet the $x_2$-axis at a right angle. 
In order to formulate the resulting initial-boundary problem, it is convenient
to rewrite (\ref{eq:newsystem}). 
To do so, we choose $\vec\nu = \vec\tau^\perp$ with the unit tangent 
$\vec\tau = \frac{\vec x_\rho}{| \vec x_\rho |}$ and $\cdot^\perp$ denoting
clockwise rotation by $\frac{\pi}{2}$. Observing that 
\[
(\vec \nu \cdot \vec e_1) \, \vec \nu = \frac{1}{| \vec x_\rho |^2} ( \vec x_\rho^\perp \cdot \vec e_1) \, \vec x_\rho^\perp = \frac{1}{| \vec x_\rho |^2}  ( \vec x_\rho \cdot \vec e_2) \,  \vec x_\rho^\perp,
\]
we are led to the following system
\begin{subequations} \label{eq:xtnew}
\begin{alignat}{2} \label{eq:xtnewflow}
\vec x_t &= 
\frac{\vec x_{\rho \rho}}{|\vec x_\rho|^2} - \frac{1}{| \vec x_\rho |^2} \, 
\frac{ \vec x_\rho \cdot \vec e_2}{\vec x \cdot \vec e_1} \,  \vec x_\rho^\perp
\quad && \text{in } (0,1) \times (0,T], \\
\vec x\cdot\vec\ek_1 & = 0, \ 
\vec x_\rho \cdot\vec\ek_2 = 0 \quad &&\text{on } \{0,1\}\times [0,T].
 \label{eq:xtbc}
\end{alignat}
\end{subequations}
Since $\vec x(\rho,t) \cdot \vec e_1 \to 0$, as $\rho \to \rho_0 \in \{0,1\}$, 
the last term in \eqref{eq:xtnewflow} needs to be treated with care.
Using the boundary conditions \eqref{eq:xtbc}, it is shown in \eqref{eq:hosp}
in Appendix~\ref{sec:A}, with the help of L'Hospital's rule, that
\begin{displaymath}
\lim_{\rho \searrow 0} \left[ - \frac{1}{| \vec x_\rho(\rho,t) |^2} \,  \frac{ \vec x_\rho(\rho,t) \cdot \vec e_2}{\vec x(\rho,t) \cdot \vec e_1} \,  \vec x_\rho^\perp(\rho,t) \right] =
\frac{\vec x_{\rho \rho} (0,t) \cdot \vec e_2}{| \vec x_\rho(0,t) |^2}  \, \vec e_2,
\end{displaymath}
so that the expression acts like a second order operator close to the boundary 
without affecting the parabolicity of the problem. Nevertheless, the different 
behaviour of $\frac{\vec\nu\cdot\vec\ek_1}{\vec x\cdot\vec\ek_1}$ in the 
interior and close to the boundary is a major problem for the analysis of a 
numerical scheme. 
Rather than using the variational form \eqref{eq:newsystemvar} that worked 
well for genus-1 surfaces, we shall introduce a scheme which directly 
discretises \eqref{eq:xtnewflow} with the help of  finite differences. 
Our main result are optimal error bounds measuring the error in discrete 
versions of the usual integral norms. 

The paper is organised as follows. In Section~\ref{sec:weak}, we formulate 
our assumptions on the solution of \eqref{eq:xtnew} and derive a number of 
properties that will be used in the error analysis. In the second part, we 
introduce our numerical scheme and provide an estimate for the consistency 
error. Section~\ref{sec:proof} is devoted to the proof of our main error 
estimates, which include an $\mathcal O(h^2+\ttau)$ bound for a 
discrete 
$H^1$--norm. 
Finally, in Section~\ref{sec:nr} we present the results of 
several numerical simulations.

We end this section with a few comments about notation. 
Throughout, $C$ denotes a generic positive constant independent of 
the mesh parameter $h$ and the time step size $\ttau$.
At times $\epsilon$ will play the role of a (small)
positive parameter, with $C_\epsilon>0$ depending on $\epsilon$, but
independent of $h$ and $\ttau$.

\setcounter{equation}{0}
\section{Finite difference discretization} \label{sec:weak}

\begin{ass} \label{ass:smooth}
Let $\vec x:[0,1] \times [0,T] \to \bRgeq\times \bR$
be a solution of \eqref{eq:xtnew} such that $\partial_t^i \partial_\rho^j \vec x$ exist and are continuous on $[0,1] \times [0,T]$ for all
$i,j \in \mathbb N_0$ with $2i+j \leq 4$. Furthermore, we assume that $\vec x_\rho(\rho,t) \neq 0$ 
for all $(\rho,t) \in [0,1] \times [0,T]$, as well as
\begin{equation} \label{eq:pos}
\vec x\cdot\vec\ek_1  > 0 \quad \text{in } (0,1) \times [0,T].
\end{equation}
\end{ass}
It is beyond the scope of this paper to prove the existence of a solution
to \eqref{eq:xtnew} with the above regularity. We note, however, that 
the well-posedness of the corresponding problem, in the
case that the curves $\Gamma(t)$ can be written as a graph,
was recently studied in \cite{GarckeM20}.

Let us collect a few 
properties of the solution which will be used in the error analysis. To begin, there exist constants $0<c_0 \leq C_0$ such that
\begin{equation} \label{eq:length}
c_0 \leq | \vec x_\rho | \leq C_0 \quad \text{in }  [0,1] \times [0,T].
\end{equation}
Recalling (\ref{eq:xtbc}), we infer that  $\vec x_\rho(0,t) \cdot \vec e_1 \geq c_0, \vec x_\rho(1,t) \cdot \vec e_1 \leq -c_0$ which together with 
(\ref{eq:pos}) implies that there exist $c_1>0, \delta >0$ with
\begin{subequations} \label{eq:leftright}
\begin{alignat}{2}
\vec x_\rho \cdot \vec e_1 &\geq \tfrac12 c_0 \quad && \text{in } 
[0,\delta] \times [0,T], \label{eq:leftbd} \\
\vec x_\rho \cdot \vec e_1 &\leq -\tfrac12 c_0 \quad && \text{in } 
[1-\delta,1] \times [0,T], \label{eq:rightbd} \\
\vec x \cdot \vec e_1 & \geq c_1 \quad && \text{in } 
[\tfrac12\delta,1-\tfrac12\delta] \times [0,T]. \label{eq:lowbound}
\end{alignat}
\end{subequations}
Let us formally describe how this observation
can be translated into an estimate on the solution. If we multiply \eqref{eq:xtnewflow} by $-\vec x_{\rho \rho}$ and integrate over $[0,1]$, 
we find 
upon integration by parts and observing from \eqref{eq:xtbc} that
$\vec x_t\cdot\vec x_\rho = 0$ on $\{0,1\}\times[0,T]$,
that
\begin{equation} \label{eq:formal1}
\tfrac12 \ddt \int_0^1 | \vec x_\rho |^2 \drho 
+ \int_0^1 \frac{| \vec x_{\rho \rho } |^2}{| \vec x_\rho |^2} \drho
- \int_0^1 \frac{1}{| \vec x_\rho |^2} \, 
\frac{ \vec x_\rho \cdot \vec e_2}{\vec x \cdot \vec e_1} \,  \vec x_\rho^\perp \cdot \vec x_{\rho \rho} \drho =0.
\end{equation}
Since \eqref{eq:xtbc} implies
$\frac{\vec x_\rho^\perp(0,t)}{| \vec x_\rho(0,t) |} \approx - \vec e_2$
on $[0,\delta]$, we can rewrite the third term on $[0,\delta]$,
on noting \eqref{eq:length} and \eqref{eq:leftbd},
as 
\begin{align} \label{eq:formal2}
& - \int_0^{\delta} \frac{1}{| \vec x_\rho |^2} \, 
\frac{ \vec x_\rho \cdot \vec e_2}{\vec x \cdot \vec e_1} \,  \vec x_\rho^\perp \cdot \vec x_{\rho \rho} \drho \\ & \
\approx \int_0^{\delta} \frac{1}{| \vec x_\rho |} \, 
\frac{ \vec x_\rho \cdot \vec e_2}{\vec x \cdot \vec e_1} \,   \vec x_{\rho \rho} \cdot \vec e_2 \drho 
= \tfrac12 \int_0^{\delta} \frac{1}{| \vec x_\rho |} \, \frac{1}{\vec x \cdot \vec e_1}
\left[ (\vec x_\rho \cdot \vec e_2)^2 \right]_\rho \drho \nonumber \\ & \
= \tfrac12 \left[ \frac{1}{| \vec x_\rho |} \, \frac{1}{\vec x \cdot \vec e_1} (\vec x_\rho \cdot \vec e_2)^2 \right]^{\delta}_0 \!\!
+ \tfrac12 \int_0^{\delta} \frac{\vec x_\rho \cdot \vec e_1}{| \vec x_\rho |} \,
\frac{(\vec x_\rho \cdot \vec e_2)^2}{(\vec x \cdot \vec e_1)^2} \drho
+ \tfrac12 \int_0^{\delta} \frac{\vec x_\rho \cdot \vec x_{\rho \rho}}{| \vec x_\rho |^3} \, \frac{(\vec x_\rho \cdot \vec e_2)^2}{\vec x \cdot
\vec e_1} \drho  \nonumber\\ & \
\geq \tfrac14 \frac{c_0}{C_0} \int_0^{\delta} \frac{(\vec x_\rho \cdot \vec e_2)^2}{(\vec x \cdot \vec e_1)^2} \drho
+ \tfrac12 \int_0^{\delta} \frac{\vec x_\rho \cdot \vec x_{\rho \rho}}{| \vec x_\rho |^3} \, \frac{(\vec x_\rho \cdot \vec e_2)^2}{\vec x \cdot
\vec e_1} \drho, \nonumber
\end{align}
so that we obtain $L^2$--control of $\frac{\vec x_\rho \cdot \vec e_2}{\vec x \cdot \vec e_1}$ close to $0$.
A similar calculation applies close to $1$, while the denominator $\vec x \cdot \vec e_1$
is bounded away from $0$ on $[\delta,1-\delta]$ in view of \eqref{eq:lowbound}. Our aim is to mimic this argument within the error analysis (cf. Lemma~\ref{lem:lemma2}).
To do so, we will directly discretise \eqref{eq:xtnewflow} using a finite difference scheme,
and the discrete analogue of the above estimate is then obtained by multiplying with a suitable second order finite difference.

In order to define our finite difference scheme,
let us introduce the set of grid points
$\mathcal G_h:= \lbrace q_0, q_1,\ldots, q_J \rbrace$, where $q_j = jh$ and $h = \frac 1J, j=0,\ldots, J$.
For a grid function $\vec v: \mathcal G_h \to \bR^2$ we write 
$\vec v_j:= \vec v(q_j)$, $j=0,\ldots,J$. 
Furthermore we associate with $\vec v$ the following finite difference 
operators:
\begin{subequations} \label{eq:fdo}
\begin{alignat}{2}
\delta^- \vec v_j &:= \frac{\vec v_j - \vec v_{j-1}}{h}, \quad &&
j=1,\ldots,J; \label{eq:fdo-} \\
\delta^+ \vec v_j &:= \delta^- \vec v_{j+1} =
\frac{\vec v_{j+1}-\vec v_j}{h}, \quad &&
j=0,\ldots,J-1; \label{eq:fdo+} \\
\delta^1 \vec v_j &:= \tfrac12 (\delta^+ \vec v_j + \delta^- \vec v_j) = 
\frac{\vec v_{j+1} - \vec v_{j-1}}{2h},\quad &&
j=1,\ldots,J-1. \label{eq:fdo1} \\
\delta^2 \vec v_j &:= \frac{\delta^+ \vec v_j - \delta^- \vec v_j}h = 
\frac{\vec v_{j+1}- 2\vec v_j + \vec v_{j-1}}{h^2}, 
\quad &&j=1,\ldots,J-1. \label{eq:fdo2}
\end{alignat}
\end{subequations}
Two grid functions $\vec v$ and $\vec w$ satisfy the following summation by
parts formula:
\begin{equation} \label{eq:sbp}
h \sum_{j=1}^J \delta^- \vec v_j \cdot \delta^- \vec w_j
= -  h \sum_{j=1}^{J-1} \vec v_j \cdot \delta^2 \vec w_j 
+ \vec v_J \cdot \delta^- \vec w_J - \vec v_0 \cdot \delta^+ \vec w_0 .
\end{equation}
In addition, we introduce the following discrete norms and seminorms
\begin{align} \label{eq:norms}
& | \vec v |_{0,h}^2 := \tfrac12 h |\vec v_0|^2 
+ h \sum_{j=1}^{J-1} |\vec v_j|^2 + \tfrac12 h |\vec v_J|^2; \quad
| \vec v |_{1,h}^2 := h \sum_{j=1}^J |\delta^- \vec v_j|^2;  \\
& \| \vec v \|_{1,h}^2 := | \vec v |_{0,h}^2 + | \vec v |_{1,h}^2; \quad
| \vec v |_{2,h}^2 := h \sum_{j=1}^{J-1} |\delta^2 \vec v_j|^2.\nonumber
\end{align}
We also recall the following inverse inequality, as well as a
discrete version of a well--known Sobolev type inequality.

\begin{lemma} \label{lem:dsi}
Let $\vec v: \mathcal G_h \to \bR^2$ be an arbitrary grid function.
Then
\begin{align} 
\max_{1\leq k \leq J} |\delta^- \vec v_k| & \leq h^{-\frac12}
| \vec v |_{1,h}, \label{eq:inv1} \\
\max_{0\leq k \leq J} | \vec v_k|^2 &
\leq | \vec v |_{0,h}^2 + 2 | \vec v |_{0,h} | \vec v |_{1,h}, \label{eq:dsi0} \\
\max_{1\leq k \leq J} |\delta^- \vec v_k|^2 &
\leq | \vec v |_{1,h}^2 + 2 | \vec v |_{1,h} | \vec v |_{2,h}. \label{eq:dsi}
\end{align}
In addition, if $\vec v_0 \cdot \vec e_1 = \vec v_J \cdot \vec e_1 =0$, then
\begin{equation} \label{eq:vje1}
| \vec v_j \cdot \vec e_1 | 
\leq 2 q_j (1-q_j) \max_{1 \leq k \leq J} | \delta^- \vec v_k |, \quad 0 \leq j \leq J. 
\end{equation}
\end{lemma}
\begin{proof}
The inverse inequality \eqref{eq:inv1} follows immediately from the definition
\eqref{eq:norms}. 
Let $0\leq k \leq J$. For $0\leq j \leq k$ it follows from \eqref{eq:fdo+}, 
the elementary inequality 
\begin{equation*} 
(a + b)^2 \leq 2(a^2 + b^2) , \quad a,b \in \bR
\end{equation*}
and \eqref{eq:norms} that
\begin{align}\label{eq:sdi0}
| \vec v_k|^2 & 
= | \vec v_j|^2 + 
\sum_{\ell=j}^{k-1} (|\vec v_{\ell+1}|^2 - | \vec v_{\ell}|^2 )
= | \vec v_j|^2 + h
\sum_{\ell=j}^{k-1} (\vec v_{\ell+1} + \vec v_{\ell} ) \cdot
\delta^+ \vec v_\ell \\ & 
\leq | \vec v_j|^2 + \sqrt{2} \Bigl( h \sum_{\ell=j}^{k-1} ( | \vec v_{\ell+1} |^2 + | \vec v_{\ell} |^2 ) \Bigr)^{\frac{1}{2}} | \vec v |_{1,h} \leq  | \vec v_j|^2 + 
2 | \vec v |_{0,h} | \vec v |_{1,h}.\nonumber 
\end{align} 
Similarly, for $k+1 \leq j \leq J$, we have 
\begin{align}
| \vec v_k|^2 & 
= | \vec v_j|^2 - 
\sum_{\ell=k}^{j-1} (| \vec v_{\ell+1}|^2 - | \vec v_{\ell}|^2 )
\leq | \vec v_j|^2 + 2 | \vec v |_{0,h} | \vec v |_{1,h}.
\label{eq:sdi1}
\end{align} 
Combining \eqref{eq:sdi0} and \eqref{eq:sdi1} yields that
$\max_{0\leq k \leq J} | \vec v_k|^2 \leq 
| \vec v_j|^2 + 2 | \vec v |_{0,h} | \vec v |_{1,h}$, for $0 \leq j \leq J$. Multiplication by $\frac{h}{2}$ for $j=0,J$, and by $h$ for $1 \leq j \leq J-1$, followed by
summation over $j=0,\ldots,J$, yields \eqref{eq:dsi0}. 
The inequality \eqref{eq:dsi} is obtained in an analogous manner,
taking into account that $\delta^+ \delta^- \vec v_j = \delta^2 \vec v_j$. 

In order to prove (\ref{eq:vje1}), we observe that 
$\vec v_0 \cdot \vec e_1= \vec v_J \cdot \vec e_1= 0$ implies
\[
| \vec v_j \cdot \vec e_1 | \leq h \sum_{k=1}^j | \delta^- \vec v_k | \leq jh \, \max_{1 \leq k \leq J} | \delta^- \vec v_k | \quad \mbox{ and } \quad | \vec v_j \cdot \vec e_1 | \leq (J-j) h \, \max_{1 \leq k \leq J} | \delta^- \vec v_k |,
\]
so that
\[
| \vec v_j \cdot \vec e_1 | \leq \min\{q_j, 1-q_j\} \, 
\max_{1 \leq k \leq J} | \delta^- \vec v_k | 
\leq 2 q_j(1-q_j) \, \max_{1 \leq k \leq J} | \delta^- \vec v_k |, 
\quad 0 \leq j \leq J.
\]
\end{proof}

We consider the following fully discrete approximation, where
in order to discretise in time, we let $t_m=m\,\ttau$, $m=0,\ldots,M$, 
with the uniform time step $\ttau = \frac TM >0$. 
Let $\vec X^0_j= \vec x_0(q_j)$, $j=0,\ldots,J$. Then, 
for $m=0,\ldots,M-1$ find $\vec X^m: \mathcal G_h \to \bR^2$ such that
for $j=1,\ldots,J-1$
\begin{subequations} \label{eq:fd}
\begin{equation} \label{eq:findiv}
\frac{\vec X^{m+1}_j - \vec X^m_j}{\Delta t} = 
\frac{\delta^2 \vec X^{m+1}_j}{| \delta^1 \vec X^m_j |^2}  
- \frac{1}{| \delta^1 \vec X^m_j |^2} \, 
 \frac{\delta^1 \vec X^{m+1}_j \cdot \vec e_2 }
{\vec X^m_j \cdot \vec  e_1}
(\delta^1 \vec X_j^m)^\perp
\end{equation}
together with the boundary conditions
\begin{align} 
\vec X^{m+1}_0 \cdot \vec e_1 & = 0; \quad 
\delta^+ \vec X^{m+1}_0 \cdot \vec e_2 = \tfrac14 h | \delta^+ \vec X^m_0 |^2
\frac{\vec X^{m+1}_0 - \vec X^m_0}{\Delta t} \cdot \vec e_2 
, \label{eq:bcl} \\
\vec X^{m+1}_J \cdot \vec e_1 & = 0; \quad 
\delta^- \vec X^{m+1}_J \cdot \vec e_2 = -\tfrac14 h | \delta^- \vec X^m_J |^2
\frac{\vec X^{m+1}_J - \vec X^m_J}{\Delta t} \cdot \vec e_2 . \label{eq:bcr}
\end{align}
\end{subequations}
The above scheme requires the solution of a linear system in each time step. 
We will address the existence and uniqueness of this system in 
Section~\ref{sec:proof}, within the error analysis.
Furthermore, we remark that \eqref{eq:bcl} and \eqref{eq:bcr} are obtained from
inserting (\ref{eq:xtnewflow}), (\ref{eq:xtbc}) into a Taylor expansion at 
$\rho \in \{0,1\}$,
yielding a consistency error that is small enough to derive optimal error 
bounds. 
At the same time, the form of these conditions turns out to be crucial in 
order to handle the degeneracy of the equation close to the axis of rotation.

\begin{lemma}[Consistency] \label{lem:consist}
Suppose that $\vec x: [0,1] \times [0,T] \to \bR^2$ satisfies 
Assumption~\ref{ass:smooth}. 
Let $\vec x^m_j:= \vec x(q_j,t_m)$ for $j=0,\ldots,J$ and
$m=0,\ldots,M$. 
Define the consistency errors of the finite difference scheme \eqref{eq:fd} by
\begin{subequations} 
\begin{equation} \label{eq:residual}
\vec R^{m+1}_j := \frac{\vec x^{m+1}_j- \vec x^m_j}{\Delta t}  - \frac{\delta^2 \vec x^{m+1}_j}{| \delta^1 \vec x^m_j |^2} 
+  \frac{1}{| \delta^1 \vec x^m_j |^2} \frac{\delta^1 \vec x^{m+1}_j \cdot \vec e_2 }{\vec x^m_j \cdot \vec e_1}  \; 
(\delta^1 \vec x_j^m)^\perp, \quad 1 \leq j \leq J-1,
\end{equation} 
as well as
\begin{align}
R^{m+1}_0 &:= \delta^+ \vec x^{m+1}_0 \cdot \vec e_2 
- \tfrac14 h | \delta^+ \vec x^m_0 |^2
\frac{ \vec x^{m+1}_0- \vec x^m_0}{\Delta t} \cdot \vec e_2 , 
\label{eq:residualb1} \\
R^{m+1}_J &:= \delta^- \vec x^{m+1}_J \cdot \vec e_2 
+ \tfrac14 h | \delta^- \vec x^m_J |^2
\frac{\vec x^{m+1}_J - \vec x^m_J}{\Delta t}  \cdot \vec e_2  . 
\label{eq:residualb2}
\end{align}
\end{subequations}
Then there exists a constant $C > 0$ such that, for $m=0,\ldots,M-1$,
\begin{equation} \label{eq:consist}
| \vec  R^{m+1}_j | \leq  C \,  \bigl( h^2 + \Delta t \bigr), \quad j=1,\ldots,J-1, 
\quad \text{ and } \quad
| R^{m+1}_0|+ | R^{m+1}_J | \leq C h  \bigl( h^2 + \Delta t \bigr).
\end{equation}
\end{lemma}
\begin{proof} 
Simple Taylor expansions yield the well-known results
\begin{subequations}
\begin{alignat}{2}
\left| \frac{\vec x^{m+1}_j- \vec x^m_j}{\Delta t} - \vec x_t(q_j,t_m) \right|
& \leq  C \Delta t, \quad && 0 \leq j \leq J,\ 0 \leq m \leq M-1, 
\label{eq:consist5} \\
| \delta^- \vec x^m_j - \vec x_\rho(q_j,t_m)  | & \leq C h,
\quad && 1 \leq j \leq J,\ 0 \leq m \leq M,  \label{eq:consist1} \\
| \delta^1 \vec x^m_j - \vec x_\rho(q_j,t_m)  | +
| \delta^2 \vec x^m_j - \vec x_{\rho \rho}(q_j,t_m)  | & \leq C h^2,
\quad && 1 \leq j \leq J-1,\ 0 \leq m \leq M, 
 \label{eq:consist2} \\
\delta^1 \vec x^m_j  - \vec x_\rho(q_j,t_m) - 
\tfrac16 h^2 \vec x_{\rho \rho \rho}(q_j,t_m) & = \mathcal O(h^3),
\quad && 1 \leq j \leq J-1,\ 0 \leq m \leq M,
\label{eq:delta1x}
\end{alignat}
\end{subequations}
where we have observed that $\vec x(q_j,\cdot) \in C^2([0,T])$ and 
$\vec x(\cdot,t_m) \in C^4([0,1])$.
Evaluating (\ref{eq:xtnewflow}) at $(\rho,t)=(q_j,t_m)$, $j=1,\ldots,J-1$, 
$m=0,\ldots,M-1$, we find that
\begin{equation} \label{eq:consistt}
\vec x_t(q_j,t_m) = \frac{\vec x_{\rho \rho}(q_j,t_m)}{|\vec x_\rho (q_j,t_m)|^2} - \frac{1}{| \vec x_\rho (q_j,t_m) |^2} \,  \frac{ \vec x_\rho (q_j,t_m) \cdot \vec e_2}{\vec x^m_j \cdot \vec e_1} \,  
\vec x_\rho^\perp (q_j,t_m),
\end{equation}
where the assumed regularity of $\vec x$ allows us to use
\eqref{eq:xtnewflow} also at time $t=0$.
If we combine \eqref{eq:residual} with \eqref{eq:consistt}, and note
\eqref{eq:consist5} as well as \eqref{eq:consist2}, we obtain
\begin{align} \label{eq:consist1a}
| \vec  R^{m+1}_j | & \leq \left| \frac{\vec x^{m+1}_j- \vec x^m_j}{\Delta t}  - \vec x_t(q_j,t_m) \right| + \left| \frac{\delta^2 \vec x^{m+1}_j}{| \delta^1 \vec x^m_j |^2} -
\frac{\vec x_{\rho \rho}(q_j,t_m)}{|\vec x_\rho (q_j,t_m)|^2} \right| 
  \\ & \qquad 
+ \frac{ | \vec x_\rho(q_j,t_m) \cdot \vec e_2 |}{ \vec x^m_j \cdot \vec e_1} \, \left| \frac{(\delta^1 \vec x_j^m)^\perp}{| \delta^1 \vec x_j^m |^2} -
\frac{ \vec x^\perp_\rho(q_j,t_m)}{| \vec x_\rho(q_j,t_m) |^2} \right| 
\nonumber  \\ & \qquad 
+ \frac{1}{| \delta^1 \vec x^m_j |} 
\frac{ | \bigl( \delta^1 \vec x^{m+1}_j  - 
\vec x_\rho(q_j,t_m) \bigr) \cdot \vec e_2 |}{ \vec x^m_j \cdot \vec e_1} \nonumber \\
& \leq C \bigl( h^2 + \Delta t \bigr) +C \frac{ | \bigl( \delta^1 \vec x^{m+1}_j - 
\vec x_\rho(q_j,t_m) \bigr) \cdot \vec e_2 |}{ \vec x^m_j \cdot \vec e_1} .
\nonumber
\end{align}
In addition, it follows from \eqref{eq:delta1x}, (\ref{eq:bc3}), 
(\ref{eq:defM}) and (\ref{eq:xje1}) that
\begin{eqnarray*}
\lefteqn{
| \bigl( \delta^1 \vec x^{m+1}_j - \vec x_\rho(q_j,t_m) \bigr) \cdot \vec e_2 | } \\
& \leq & | \bigl( \delta^1 \vec x^{m+1}_j - \delta^1 \vec x^m_j \bigr) \cdot \vec e_2 | + | \bigl( \delta^1 \vec x^m -  \vec x_\rho(q_j,t_m) \bigr) \cdot \vec e_2 | \\
& \leq & \Delta t \sup_{t_m \leq t \leq t_{m+1}} | \delta^1 \vec x_t(q_j,t) \cdot \vec e_2 | +  C h^2 \min_{q \in \{0,1 \}} 
| \vec x_{\rho \rho \rho}(q_j,t_m) \cdot \vec e_2 
- \vec x_{\rho \rho \rho}(q,t_m) \cdot \vec e_2 | + C h^3 \\
& \leq & K \Delta t \sup_{t_m \leq t \leq t_{m+1}} \vec x(q_j,t) \cdot \vec e_1 + C h^2 \min_{q \in \{0,1 \}} | \vec x_{\rho \rho \rho}(q_j,t_m) - \vec x_{\rho \rho \rho}(q,t_m) | + C h^3 \\ 
& \leq &  C \Delta t q_j(1-q_j) +  C h^2 \min_{q \in \{0,1 \}} |q_j - q| + C h^3
\leq C q_j (1-q_j) \bigl(\Delta t +  h^2\bigr).
\end{eqnarray*}
If we insert this bound into (\ref{eq:consist1a}) and note that 
$\vec x^m_j \cdot \vec e_1 \geq c_2 q_j (1-q_j)$, $0 \leq j \leq J$, 
in view of (\ref{eq:xje1}), 
we obtain (\ref{eq:consist}) for $\vec R^{m+1}_j$, $j=1,\ldots,J-1$.
Let us next examine $R^{m+1}_0$. A Taylor expansion yields
\begin{displaymath}
\delta^+ \vec x^{m+1}_0 = \frac{\vec x^{m+1}_1 - \vec x^{m+1}_0}{h} = \vec x_\rho(0,t_{m+1}) + \tfrac12 h \vec x_{\rho \rho}(0,t_{m+1}) + \tfrac16 h^2
\vec x_{\rho \rho \rho}(0,t_{m+1}) + \mathcal O(h^3),
\end{displaymath}
which together with (\ref{eq:xtbc}), \eqref{eq:bc3}, 
$\vec x_{\rho \rho}(0,\cdot) \in C^1([0,T])$
and (\ref{eq:speedbd}) implies that
\begin{align*}
\delta^+ \vec x^{m+1}_0 \cdot \vec e_2 & =
 \tfrac12 h \vec x_{\rho \rho}(0,t_{m+1}) \cdot \vec e_2 + \mathcal{O}(h^3) =   \tfrac12 h \vec x_{\rho \rho}(0,t_m) \cdot \vec e_2 + \mathcal{O} \bigl( h (h^2 + \Delta t) \bigr)\\
& = \tfrac14 h (\vec x_t(0,t_m) \cdot \vec e_2) | \vec x_{\rho}(0,t_m) |^2 + \mathcal{O} \bigl( h(h^2+ \Delta t) \bigr) \\
& = \tfrac14 h \, \frac{\vec x^{m+1}_0 - \vec x^m_0}{\Delta t}  \cdot \vec e_2 |  \delta^+ \vec x^m_0 |^2 +  \mathcal{O} \bigl( h(h^2+ \Delta t) \bigr),
\end{align*}
where we have used (\ref{eq:consist5})
as well as
\begin{align*}
| \delta^+ \vec x^m_0 |^2 - | \vec x_{\rho}(0,t_m) |^2 & 
= \bigl( \delta^+ \vec x^m_0 - \vec x_{\rho}(0,t_m) \bigr) \cdot  \bigl( \delta^+ \vec x^m_0 + \vec x_{\rho}(0,t_m) \bigr) \\ & 
= \bigl( \tfrac12 h \vec x_{\rho \rho}(0,t_m) + \mathcal O(h^2) \bigr) \cdot \bigl( 2 \vec x_{\rho}(0,t_m) + \mathcal O(h) \bigr)  \\ &  
= h \vec x_{\rho \rho}(0,t_m) \cdot
\vec x_{\rho}(0,t_m) + \mathcal O(h^2)= \mathcal O(h^2),
\end{align*}
recall \eqref{eq:bc2}.
The bound for $R^{m+1}_J$ is obtained in a similar way.
\end{proof}

\begin{theorem} \label{thm:main} 
Suppose that $\vec x: [0,1] \times [0,T] \to \bR^2$ satisfies 
Assumption~\ref{ass:smooth}.
Then there exist $h_0>0, \gamma>0$ such that the discrete solution 
$(\vec X^m)_{m=1,\ldots,M}$ to \eqref{eq:fd} exists, and the error 
\begin{equation} \label{eq:defEj}
\vec E^m_j:= \vec x^m_j- \vec X^m_j, \quad j=0,\ldots,J;\ m=0, \ldots,M
\end{equation}
satisfies:
\begin{align} 
\max_{1 \leq m \leq M} \left[  \| \vec E^m \|_{1,h}^2  + \max_{0 \leq j \leq J} | \vec E^m_j |^2
\right] &\leq C \bigl( h^4 + (\Delta t)^2 \bigr), \label{eq:main} \\
 \Delta t \sum_{m=1}^M \left[
 | \vec E^m |_{2,h}^2 + \left| \frac{ \vec E^m -\vec E^{m-1}}{\Delta t} \right|_{0,h}^2 
\right] & \leq C \bigl( h^4 + (\Delta t)^2 \bigr), \label{eq:maindt} 
\end{align}
provided that $0 < h \leq h_0$ and $\Delta t \leq \gamma h$.
\end{theorem}

\setcounter{equation}{0}
\section{Proof of Theorem \ref{thm:main}} \label{sec:proof}

Assumption~\ref{ass:smooth} assures the existence of positive constants
$c_0, C_0, c_1, \delta$ such that \eqref{eq:length} and \eqref{eq:leftright}
hold.
Let $h \leq \delta$.
We set $J_1:= \left \lfloor\frac{\delta}{h} \right\rfloor \in \bZ_{\geq1}$,
so that $q_{J_1} = J_1 h \in [\frac12\delta,\delta]$. We shall prove Theorem \ref{thm:main} with the help of an induction argument. 
In particular, we will prove that there exist $h_0 >0$, $0< \gamma \leq 1$ and 
$\mu>0$ such that if $0<h \leq h_0$ and $\Delta t \leq \gamma h$, 
then for $m \in \{ 0,\ldots,M \}$ the discrete solution $\vec X^m$ exists 
and satisfies
\begin{equation} \label{eq:indass}
\| \vec E^m \|_{1,h}^2  \leq \bigl( h^4 + (\Delta t)^2 \bigr) e^{\mu t_m}.
\end{equation}
The assertion \eqref{eq:indass} clearly holds for $m=0$,
for arbitrary $h_0 \leq \delta$, $0 < \gamma \leq 1$ and $\mu>0$.
On assuming that \eqref{eq:indass} holds for a fixed 
$m \in \{ 0,\ldots,M-1 \}$, we will now show that it also holds for $m+1$.

To begin, let us choose $0<h_0\leq\delta$ and $0< \gamma \leq 1$ so small that 
\begin{equation*} 
(h_0^2 + \gamma^2) e^{\mu T} \leq 1.
\end{equation*}
Then, since $\Delta t \leq \gamma h$, (\ref{eq:indass}) implies that 
\begin{equation*} 
\|  \vec E^m \|_{1,h}^2  \leq h^2 ( h^2 + \gamma^2) e^{\mu t_m} \leq h^{2}, \qquad 0<h \leq h_0.
\end{equation*}
In particular, we infer from Lemma~\ref{lem:dsi} that
\begin{equation} \label{eq:maxest}
\max_{0 \leq j \leq J} | \vec E^m_j | 
+ \max_{1 \leq j \leq J} | \delta^- \vec E^m_j | 
+ \max_{1 \leq j \leq J-1} | \delta^1 \vec E^m_j| \leq C h^{\frac12}.
\end{equation}
This implies for $1 \leq j \leq  J$, on recalling 
\eqref{eq:consist1} and \eqref{eq:length}, that
\[
| \delta^- \vec X^m_j | \leq | \delta^- \vec x^m_j | + | \delta^- \vec E^m_j | 
\leq | \vec x_\rho(q_j,t_m) | + C h^{\frac12} 
\leq C_0 + C h^{\frac12},
\]
and similarly $| \delta^- \vec X^m_j | \geq c_0 - Ch^{\frac12}$. 
Arguing in the same way for $\delta^1 \vec X^m_j$, we infer that
\begin{equation} \label{eq:lengthh}
\tfrac12 c_0 \leq | \delta^- \vec X^m_j | \leq 2 C_0, \; 1 \leq j \leq J;  \quad \tfrac12 c_0 \leq  | \delta^1 \vec X^m_j |  \leq 2 C_0, \;
1 \leq j  \leq J-1,
\end{equation}
provided that $0<h \leq h_0$ and $h_0>0$ is chosen smaller if necessary. A similar argument together with (\ref{eq:leftbd}), (\ref{eq:rightbd}) shows that
\begin{equation}  \label{eq:xrhoe1}
\delta^- \vec X^m_j \cdot \vec e_1 \geq \tfrac14 c_0, \quad 1 \leq j \leq J_1; \qquad \delta^- \vec X^m_j \cdot \vec e_1  \leq - \tfrac14 c_0, \quad
J-J_1 \leq j \leq J.
\end{equation}
Next, since $\vec x_\rho(0,t_m) \cdot \vec e_2=0$, recall \eqref{eq:xtbc}, 
we have from \eqref{eq:maxest} and \eqref{eq:bcl} that 
\begin{align*}
h \, | \delta^- \vec X^m_1 | & 
\leq h \, | \delta^- \vec x^m_1 | 
+ h \, | \delta^- \vec E^m_1 | 
\leq h \, \delta^- \vec x^m_1 \cdot \vec e_1 + h \, 
| \delta^- \vec x^m_1 \cdot \vec e_2 | + Ch^{\frac32} \\ & 
\leq h \, \delta^- \vec X^m_1 \cdot \vec e_1 + C h^{\frac32} 
= \vec X^m_1 \cdot \vec e_1 + Ch^{\frac32} \leq (1+ C h^{\frac12}) \vec x^m_1 \cdot \vec e_1 + Ch^{\frac32},
\end{align*}
where in the last step we have observed that
\[
\vec X^m_1 \cdot \vec e_1 - \vec x^m_1 \cdot \vec e_1 
= -h \delta^- \vec E^m_1 \cdot \vec e_1 
\leq C h^{\frac32} \leq Ch^{\frac12} \vec x^m_1 \cdot \vec e_1,
\]
on noting $\vec x^m_1 \cdot \vec e_1 \geq \tfrac12 c_0 h$.
Arguing in the same way at the right boundary, we obtain
\begin{equation} \label{eq:bdh}
\tfrac34 h \, | \delta^- \vec X^m_1 | \leq \vec X^m_1 \cdot \vec e_1 \leq \tfrac43 \vec x^m_1 \cdot \vec e_1; \qquad 
\tfrac34 h | \delta^- \vec X^m_J |  \leq \vec X^m_{J-1} \cdot \vec e_1 
\leq \tfrac43 \vec x^m_{J-1} \cdot \vec e_1,
\end{equation}
for a possibly smaller $h_0>0$. 
Next, (\ref{eq:maxest}), \eqref{eq:lowbound} and the fact that $J_1h \geq \tfrac12\delta$
imply that 
\begin{equation} \label{eq:xmje1}
\vec X^m_j \cdot \vec e_1 \geq c_1 - Ch^{\frac12} \geq \tfrac12 c_1, 
\quad J_1 \leq j \leq J-J_1,
\end{equation}
after choosing $h_0$ again smaller if required. In addition, there exists $c_3>0$ such that
\begin{equation}  \label{eq:Xje1}
\vec X^m_j \cdot \vec e_1 \geq c_3 q_j (1-q_j), \quad 0 \leq j \leq J.
\end{equation}
To see this, note that
$\vec X^m_0 \cdot \vec e_1 = 0$ and (\ref{eq:xrhoe1}) imply that
\begin{subequations} 
\begin{equation} \label{eq:xXjlowerleft}
\vec X^m_j \cdot \vec e_1 \geq \tfrac14 c_0 jh \geq \tfrac14 c_0 q_j (1- q_j), \quad 0 \leq j \leq J_1, 
\end{equation}
and similarly 
\begin{equation} \label{eq:xXjlowerright}
\vec X^m_j \cdot \vec e_1 \geq \tfrac14 c_0 q_j (1 - q_j),\quad
J-J_1 \leq j \leq J.
\end{equation}
\end{subequations}
Combining these estimates with (\ref{eq:xmje1}) proves the bound 
(\ref{eq:Xje1}). 
If we combine (\ref{eq:Xje1}) with (\ref{eq:vje1}) and (\ref{eq:lengthh}), 
we obtain 
\begin{equation} \label{eq:quotbd}
\frac{\vec X^m_{j \pm 1} \cdot \vec e_1}{\vec X^m_j \cdot e_1} \leq \frac{4C_0}{c_3} \frac{q_{j \pm 1}(1-q_{j \pm 1})}{q_j (1-q_j)} \leq \frac{8 C_0}{c_3}. 
\quad 1 \leq j \leq J-1.
\end{equation}
Finally,  (\ref{eq:length}) and (\ref{eq:lengthh}) imply that
\begin{subequations}
\begin{alignat}{2}
\left| \frac{1}{| \delta^- \vec x^m_j |^2} - \frac{1}{| \delta^- \vec X^m_j |^2} \right| & \leq C | \delta^- \vec E^m_j |,
&& \quad 1 \leq j \leq J, \label{eq:esta-} \\
\left| \frac{1}{| \delta^1 \vec x^m_j |^2} - \frac{1}{| \delta^1 \vec X^m_j |^2} \right| & \leq C | \delta^1 \vec E^m_j |,
&&\quad 1 \leq j \leq J-1. \label{eq:esta}
\end{alignat}
\end{subequations}

\begin{lemma} [Existence and uniqueness] \label{lem:ex}
Let $\vec X^m: \mathcal G_h \rightarrow \bR^2$ be as above. 
Then \eqref{eq:fd} has a unique solution $\vec X^{m+1}: \mathcal G_h \rightarrow \bR^2$, provided that $h_0$ is small enough.
\end{lemma}
\begin{proof}
The relations (\ref{eq:fd}) form a linear system with $2(J+1)$ unknowns for the $2(J+1)$ values of $\vec X^{m+1}$ at the nodes $q_j$, $j=0,\ldots,J$. It is 
therefore sufficient to show that the corresponding homogeneous system
\begin{subequations} 
\begin{align}
& \frac{1}{\Delta t} \vec X_j - \frac{1}{| \delta^1 \vec X^m_j |^2} \delta^2 \vec X_j =- \frac{1}{| \delta^1 \vec X^m_j |^2} \frac{\delta^1 \vec X_j \cdot \vec e_2}{\vec X^m_j \cdot \vec e_1}
(\delta^1 \vec X^m_j)^\perp,  \quad j=1,\ldots,J-1;  \label{eq:ex1a} \\
& \vec X_0 \cdot \vec e_1 = 0; \quad 
\delta^+ \vec X_0 \cdot \vec e_2 = \tfrac14 \frac{h}{\Delta t} 
(\vec X_0 \cdot \vec e_2) \, | \delta^+ \vec X^m_0 |^2; \label{eq:ex1b} \\
& \vec X_J \cdot \vec e_1  = 0; \quad 
\delta^- \vec X_J \cdot \vec e_2 = -\tfrac14 \frac{h}{\Delta t} 
(\vec X_J \cdot \vec e_2) \, | \delta^- \vec X^m_J |^2 \label{eq:ex1c}
\end{align}
\end{subequations}
only has the trivial solution $\vec X=\vec 0$. If we multiply (\ref{eq:ex1a}) 
with $- h \delta^2 \vec X_j$ and sum from $j=1,\ldots,J-1$ we 
obtain with the help of (\ref{eq:sbp}) that
\begin{align*}
& \frac{1}{\Delta t} | \vec X |_{1,h}^2 + \frac{1}{\Delta t} \bigl( \vec X_0 \cdot \delta^+ \vec X_0 - \vec X_J \cdot \delta^- \vec X_J \bigr) +
h \sum_{j=1}^{J-1} \frac{1}{| \delta^1 \vec X^m_j |^2} | \delta^2 \vec X_j |^2 \\ & \quad
= h \sum_{j=1}^{J-1} \frac{1}{| \delta^1 \vec X^m_j |^2} \frac{\delta^1 \vec X_j \cdot \vec e_2}{\vec X^m_j \cdot \vec e_1} (\delta^1 \vec X^m_j)^\perp \cdot \delta^2 \vec X_j.
\end{align*}
In view of (\ref{eq:ex1b}) and (\ref{eq:bdh}) we have
\begin{align*}
\frac{1}{\Delta t} \vec X_0 \cdot \delta^+ \vec X_0 = \frac{1}{\Delta t}   (\vec X_0 \cdot \vec e_2) (\delta^+ \vec X_0 \cdot \vec e_2) = 
\frac{4}{h} \frac{(\delta^+ \vec X_0 \cdot \vec e_2)^2}{| \delta^+ \vec X^m_0 |^2} \geq \tfrac94 h \frac{(\delta^- \vec X_1 \cdot \vec e_2)^2}{( \vec X^m_1 \cdot \vec e_1)^2}.
\end{align*}
Using a similar argument at the right end point, as well as \eqref{eq:lengthh}, we deduce 
\begin{align}
& \frac{1}{\Delta t} | \vec X |_{1,h}^2 + \frac{1}{4 C_0^2} | \vec X |_{2,h}^2 + \tfrac94 h \Bigl(  \frac{(\delta^- \vec X_1 \cdot \vec e_2)^2}{( \vec X^m_1 \cdot \vec e_1)^2}
+ \frac{(\delta^+ \vec X_{J-1} \cdot \vec e_2)^2}{( \vec X^m_{J-1} \cdot \vec e_1)^2} \Bigr) \label{eq:ex1d} \\ & \quad 
\leq h \sum_{j=1}^{J-1} \frac{\delta^1 \vec X_j \cdot \vec e_2}{\vec X^m_j \cdot \vec e_1} \Bigl( \frac{(\delta^1 \vec X^m_j)^\perp}{| \delta^1 \vec X^m_j |^2}  
- \frac{(\delta^1 \vec x^m_j)^\perp}{| \delta^1 \vec x^m_j |^2} \Bigr) \cdot \delta^2 \vec X_j \nonumber \\ & \qquad
+ h \sum_{j=1}^{J-1} \frac{1}{| \delta^1 \vec x^m_j |^2} 
\frac{\delta^1 \vec X_j \cdot \vec e_2}{\vec X^m_j \cdot \vec e_1} (\delta^1 \vec x^m_j)^\perp \cdot \delta^2 \vec X_j  \nonumber \\ & \quad
=: h \sum_{j=1}^{J-1} \vec S^1_j \cdot \delta^2 \vec X_j + h \sum_{j=1}^{J-1} \vec S^2_j \cdot \delta^2 \vec X_j . \nonumber
\end{align}
Using (\ref{eq:esta}) and (\ref{eq:maxest}) we infer that
\begin{displaymath}
| \vec S^1_j | \leq C \frac{| \delta^1 \vec X_j \cdot \vec e_2 |}{ \vec X^m_j \cdot \vec e_1} \, | \delta^1 \vec E^m_j | \, | \delta^2 \vec X_j |
\leq  C h^{\frac{1}{2}} \frac{| \delta^1 \vec X_j \cdot \vec e_2 |}{ \vec X^m_j \cdot \vec e_1} \, | \delta^2 \vec X_j |
\end{displaymath}
and hence
\begin{displaymath}
h \sum_{j=1}^{J-1} \vec S^1_j \cdot \delta^2 \vec X_j
 \leq \frac{1}{8 C_0^2} | \vec X |_{2,h}^2 + C h^2 \sum_{j=1}^{J-1} \frac{( \delta^1 \vec X_j \cdot \vec e_2)^2}{( \vec X^m_j \cdot \vec e_1)^2}.
\end{displaymath}
The term $\vec S^2_j$ corresponds exactly to $- \vec T^{m,3}_j$ in (\ref{eq:err1}) below, if we replace $\vec E^{m+1}$ by $\vec X$. We may therefore
deduce from  Lemma \ref{lem:lemma2} that
\begin{align*}
 h \sum_{j=1}^{J-1} \vec S^2_j \cdot \delta^2 \vec X_j & \leq - c_4 h \, \sum_{j=1}^{J-1} \frac{ ( \delta^1 \vec X_j  \cdot \vec e_2)^2}
{(\vec X^m_j \cdot \vec e_1)^2}
+ \tfrac53 h \left( \frac{( \delta^- \vec X_1  \cdot \vec e_2)^2}{(\vec X^m_1 \cdot \vec e_1)^2} + 
\frac{( \delta^+ \vec X_{J-1}  \cdot \vec e_2)^2}{(\vec X^m_{J-1} \cdot \vec e_1)^2} \right) 
\nonumber \\ & \qquad
+ \frac{1}{8 C_0^2}  | \vec X |_{2,h}^2 + C \, | \vec X |_{1,h}^2.
\nonumber
\end{align*}
If we insert the above bounds into (\ref{eq:ex1d}) and recall that $\Delta t \leq \gamma h \leq h$ we infer that
\begin{displaymath}
\bigl( \frac{1}{h} - C \bigr) | \vec X |_{1,h}^2 + (c_4 - Ch) h  \sum_{j=1}^{J-1} \frac{ ( \delta^1 \vec X_j  \cdot \vec e_2)^2}
{(\vec X^m_j \cdot \vec e_1)^2} \leq 0,
\end{displaymath}
which implies that $\vec X \equiv \vec X_0$ provided that $0< h \leq h_0$, where $h_0$ is chosen smaller if necessary. 
The boundary conditions (\ref{eq:ex1b}), on noting \eqref{eq:lengthh}, then
yield $\vec X \equiv 0$.
\end{proof}
 
We begin our error analysis by 
combining \eqref{eq:defEj}, (\ref{eq:findiv}) and 
(\ref{eq:residual}), in order to derive the following error relation:
\begin{align}  \label{eq:err1}
 & \frac{ \vec E^{m+1}_j - \vec E^m_j}{\Delta t}   
- \frac{\delta^2 \vec E^{m+1}_j}{| \delta^1 \vec X^m_j |^2} 
= \Bigl( \frac{1}{| \delta^1 \vec x^m_j |^2}- \frac{1}{| \delta^1 \vec X^m_j |^2} \Bigr)
\delta^2 \vec x^{m+1}_j  \\ & \qquad
+ \frac{\delta^1 \vec X^{m+1}_j \cdot \vec e_2}{\vec X^m_j \cdot \vec  e_1} \left[ \Bigl( \frac{1}{| \delta^1 \vec X^m_j |^2} - \frac{1}{| \delta^1 \vec x^m_j |^2} \Bigr)
\; (\delta^1 \vec X_j^m)^\perp
 - \frac{1}{| \delta^1 \vec x^m_j |^2} \; (\delta^1 \vec E_j^m)^\perp \right] 
\nonumber \\ & \qquad
- \frac{1}{| \delta^1 \vec x^m_j |^2} \, \frac{\delta^1 \vec E^{m+1}_j \cdot \vec e_2 }{\vec X^m_j \cdot \vec  e_1} \; (\delta^1 \vec x_j^m)^\perp \nonumber \\ & \qquad
+ \frac{1}{| \delta^1 \vec x^m_j |^2 } \, 
\Bigl( \frac{1}{\vec X^m_j \cdot \vec e_1} - \frac{1}{\vec x^m_j \cdot \vec e_1} 
\Bigr) \; ( \delta^1 \vec x^{m+1}_j \cdot \vec e_2) (\delta^1 \vec x_j^m)^\perp 
 + \vec R^{m+1}_j \nonumber \\ & \quad
=: \sum_{i=1}^{5} \vec T^{m,i}_j,  
\qquad 1 \leq j \leq J-1. \nonumber
\end{align}
Furthermore, for the boundary points we have in view of 
\eqref{eq:xtbc}, \eqref{eq:bcl}, \eqref{eq:bcr}, \eqref{eq:residualb1} 
and \eqref{eq:residualb2} that
\begin{subequations} 
\begin{align}
& (  \vec E^{m+1}_0  - \vec E^m_0) \cdot \vec e_1  = (  \vec E^{m+1}_J - \vec E^m_J) \cdot \vec e_1 =0,  \label{eq:dtE0} \\
& \frac{\vec E^{m+1}_0 - \vec E^m_0}{\Delta t}  = \frac{4}{h} 
\frac{\delta^+ \vec E^{m+1}_0 \cdot \vec e_2 - R^{m+1}_0}{| \delta^+ \vec X^m_0 |^2}\vec e_2 
+ \Bigl( 1 - \frac{| \delta^+ \vec x^m_0 |^2}{| \delta^+ \vec X^m_0|^2} \Bigr) 
\frac{\vec x^{m+1}_0 - \vec x^m_0}{\Delta t} 
, \label{eq:err1l} \\
& \frac{\vec E^{m+1}_J - \vec E^m_J}{\Delta t}  = -\frac{4}{h} 
\frac{\delta^- \vec E^{m+1}_J \cdot \vec e_2 - R^{m+1}_J}{| \delta^- \vec X^m_J |^2} \vec e_2
+ \Bigl( 1 - \frac{| \delta^- \vec x^m_J |^2}{| \delta^- \vec X^m_J|^2} \Bigr) 
\frac{ \vec x^{m+1}_J - \vec x^m_J}{\Delta t} 
. \label{eq:err1r} 
\end{align}
\end{subequations}

Our strategy for the proof of \eqref{eq:indass} with $m$ replaced by $m+1$
is now as follows. In a
discrete analogue to the formal procedure in \eqref{eq:formal1}, we are going
to multiply \eqref{eq:err1} with a second order difference of the error 
$\vec E^{m+1}$. The ensuing analysis is technical, and so we split it into
three steps. In a first step, we control the terms generated on the left hand
side of \eqref{eq:err1}, in order to obtain Lemma~\ref{lem:lemma1a}. Next we
estimate four of the five terms generated by the right hand side of 
\eqref{eq:err1}, see Lemma~\ref{lem:lemma1b}. The remaining term,
which is generated by $\vec T^{m,3}$ and loosely corresponds to the last 
integral in \eqref{eq:formal1}, requires a particularly careful analysis.
We present the derived estimate in Lemma~\ref{lem:lemma2}, where in the proof 
we will mimic the formal calulations from \eqref{eq:formal2}. 

The induction step is then completed by combining the three lemmas.

\begin{lemma} \label{lem:lemma1a} 
There exists $C_1 > 0$ such that for all $0 < \lambda \leq 1$
\begin{align}\label{eq:err1a}
& \frac{1+\lambda}{2 \Delta t} \bigl( | \vec E^{m+1} |_{1,h}^2 - | \vec E^m |_{1,h}^2 \bigr) + \frac{1}{2 \Delta t} | \vec E^{m+1} - \vec E^m |_{1,h}^2 + \frac{1}{4 C_0^2} | \vec E^{m+1} |_{2,h}^2  \\ & \qquad 
+ \tfrac14 c_0^2 \lambda \, 
\left|  \frac{\vec E^{m+1}- \vec E^m}{\Delta t} \right|^2_{0,h} 
+ (2 - C_1 \lambda )h \left[ \frac{\bigl( \delta^- \vec E^{m+1}_1 \cdot \vec e_2 \bigr)^2}{( \vec X^m_1 \cdot \vec e_1)^2} +  \frac{\bigl( \delta^+ \vec E^{m+1}_{J-1} \cdot \vec e_2 \bigr)^2}{( \vec X^m_{J-1} \cdot \vec e_1)^2}  \right] 
\nonumber \\ & \quad 
\leq C h \bigl( h^4 + (\Delta t)^2 \bigr) + C | \vec E^m |_{1,h}^2
\nonumber \\ & \qquad \quad
+ h \sum_{i=1}^{5} \sum_{j=1}^{J-1} \vec T^{m,i}_j  \cdot \bigl( \lambda | \delta^1 \vec X^m_j |^2  \frac{\vec E^{m+1}_j- \vec E^m_j}{\Delta t} - \delta^2 \vec E^{m+1}_j \bigr). \nonumber
\end{align}
\end{lemma}
\begin{proof} Fix $0< \lambda \leq 1$. 
If we multiply (\ref{eq:err1}) by $h  \bigl( \lambda | \delta^1 \vec X^m_j |^2  \frac{\vec E^{m+1}_j- \vec E^m_j}{\Delta t} - \delta^2 \vec E^{m+1}_j \bigr)$ and sum over 
$j=1,\ldots,J-1$, we obtain
\begin{align} \label{eq:err2}
& -(1+ \lambda)  \frac{h}{\Delta t} \sum_{j=1}^{J-1} \bigl( \vec E^{m+1}_j - \vec E^m_j \bigr) \cdot \delta^2 \vec E^{m+1}_j + 
 h \sum_{j=1}^{J-1} \frac{ | \delta^2 \vec E^{m+1}_j |^2 }{| \delta^1 \vec X^m_j |^2} \\ & \qquad\qquad
+h \lambda \sum_{j=1}^{J-1} | \delta^1 \vec X^m_j |^2 \left| \frac{ \vec E^{m+1}_j - \vec E^m_j}{\Delta t} \right|^2  \nonumber \\ & \quad
= h \sum_{i=1}^{5} \sum_{j=1}^{J-1} \vec T^{m,i}_j  \cdot \bigl( \lambda | \delta^1 \vec X^m_j |^2  \frac{\vec E^{m+1}_j- \vec E^m_j}{\Delta t} - \delta^2 \vec E^{m+1}_j \bigr). \nonumber 
\end{align}
Applying summation by parts, \eqref{eq:sbp}, 
to the first term in \eqref{eq:err2}, and noting \eqref{eq:norms} and 
$2(a-b)a = a^2 - b^2 + (a-b)^2$, yields
\begin{align} \label{eq:err2a}
& - \frac{h}{\Delta t}  \sum_{j=1}^{J-1} \bigl( \vec E^{m+1}_j - \vec E^m_j \bigr) \cdot \delta^2 \vec E^{m+1}_j  \\ & \quad
= \frac{h}{\Delta t}  \sum_{j=1}^J \delta^- \bigl(\vec E^{m+1}_j - \vec E^m_j \bigr)  \cdot \delta^- \vec E^{m+1}_j 
\nonumber \\ & \qquad \quad
- \frac{\vec E^{m+1}_J - \vec E^m_J}{\Delta t}  \cdot \delta^- \vec E^{m+1}_J
+ \frac{\vec E^{m+1}_0 - \vec E^m_0}{\Delta t}  \cdot \delta^+ \vec E^{m+1}_0 \nonumber \\ & \quad
= \frac{1}{2 \Delta t} \bigl( | \vec E^{m+1} |_{1,h}^2 - | \vec E^m |_{1,h}^2 \bigr) + \frac{1}{2 \Delta t} | \vec E^{m+1} - \vec E^m |_{1,h}^2 
\nonumber \\ & \qquad \quad
- \frac{\vec E^{m+1}_J - \vec E^m_J}{\Delta t}  \cdot \delta^- \vec E^{m+1}_J
+ \frac{\vec E^{m+1}_0 - \vec E^m_0}{\Delta t}  \cdot \delta^+ \vec E^{m+1}_0. \nonumber
\end{align}
On noting (\ref{eq:err1l}), \eqref{eq:fdo+}, 
Young's inequality, 
(\ref{eq:consist}), \eqref{eq:lengthh}, \eqref{eq:esta-}
and (\ref{eq:bdh}), we can 
estimate the last term on the right hand side of \eqref{eq:err2a} as 
\begin{align*} 
& \frac{\vec E^{m+1}_0 - \vec E^m_0}{\Delta t}  \cdot \delta^+ \vec E^{m+1}_0
 = \frac{4}{h} \frac{\bigl( \delta^+ \vec E^{m+1}_0 \cdot \vec e_2 \bigr)^2}
{| \delta^+ \vec X^m_0 |^2} - \frac{4}{h} \frac{\delta^+ \vec E^{m+1}_0 \cdot \vec e_2}{| \delta^+ \vec X^m_0 |^2} R^{m+1}_0  
 \\ & \hspace{4cm}
+ \Bigl( 1 - \frac{| \delta^+ \vec x^m_0 |^2}{| \delta^+ \vec X^m_0 |^2} \Bigr) 
\frac{ \vec x^{m+1}_0 - \vec x^m_0}{\Delta t} \cdot \vec e_2 (\delta^+ \vec E^{m+1}_0 \cdot \vec e_2)  \nonumber \\ & \quad
\geq \frac{4 - \epsilon}{h} 
\frac{\bigl(\delta^+ \vec E^{m+1}_0 \cdot \vec e_2 \bigr)^2}{| \delta^- \vec X^m_1 |^2}
- C_\epsilon h  \bigl( h^4 + (\Delta t)^2 \bigr) 
- C_\epsilon h \, | \delta^+ \vec E^m_0 |^2  
\nonumber \\ & \quad
\geq \tfrac{9}{16}(4-\epsilon)h
\, \frac{\bigl( \delta^+ \vec E^{m+1}_0 \cdot \vec e_2 \bigr)^2}{( \vec X^m_1 \cdot \vec e_1)^2} 
- C_\epsilon h \bigl( h^4 + 
(\Delta t)^2 \bigr)
- C_\epsilon h \, | \delta^+ \vec E^m_0 |^2 . \nonumber
\end{align*}
On choosing $\epsilon$ sufficiently small, and arguing similarly for 
$\frac{ \vec E^{m+1}_J - \vec E^m_J}{\Delta t}  \cdot  \delta^- \vec E^{m+1}_J$, 
we find that \eqref{eq:err2a} implies
\begin{align}  \label{eq:err3}
& -(1+ \lambda)  \frac{h}{\Delta t}  \sum_{j=1}^{J-1} \bigl( \vec E^{m+1}_j - \vec E^m_j \bigr) \cdot \delta^2 \vec E^{m+1}_j
\geq  \ \frac{1+ \lambda}{2 \Delta t} \bigl( | \vec E^{m+1} |_{1,h}^2 - | \vec E^m |_{1,h}^2 \bigr) \\ & \qquad
 + \frac{1}{2 \Delta t} | \vec E^{m+1}- \vec E^m |_{1,h}^2
+ 2 h \left[ \frac{\bigl( \delta^+ \vec E^{m+1}_0 \cdot \vec e_2 \bigr)^2}{( \vec X^m_1 \cdot \vec e_1)^2} +  \frac{\bigl( \delta^- \vec E^{m+1}_J \cdot \vec e_2 \bigr)^2}{( \vec X^m_{J-1} \cdot \vec e_1)^2}
\right]  
\nonumber \\ & \qquad
- C h \bigl( | \delta^- \vec E^m_1 |^2 + | \delta^- \vec E^m_J |^2 \bigr) - C h \bigl( h^4 + (\Delta t)^2 \bigr).\nonumber
\end{align}
In addition, we deduce from (\ref{eq:err1l}), \eqref{eq:esta-}, 
\eqref{eq:lengthh} and the fact that thanks to \eqref{eq:Xje1} we have
$h |\delta^+ \vec X^m_0|^2 \geq h^{-1}\,(\vec X^m_1\cdot\vec e_1)^2
\geq C\,\vec X^m_1\cdot\vec e_1$
that
\begin{align*}
\left| \frac{ \vec E^{m+1}_0 - \vec E^m_0 }{\Delta t} \right| 
 \leq C   \frac{ | \delta^- \vec E^{m+1}_1  \cdot \vec e_2|}{\vec X^m_1 \cdot \vec e_1} 
 + C  \frac1h  | R^{m+1}_0 | + C  | \delta^- \vec E^m_1 | ,  
\end{align*}
so that \eqref{eq:consist} yields
\begin{align}
\frac{h}{2} \left| \frac{ \vec E^{m+1}_0 - \vec E^m_0 }{\Delta t} \right|^2 \leq  C h   \frac{ (\delta^-\vec E^{m+1}_1 \cdot \vec e_2 )^2}{(\vec X^m_1 \cdot \vec e_1)^2}
+ C h ( h^2 + \Delta t)^2 + C h  | \delta^- \vec E^m_1 |^2.  \label{eq:bddt}
\end{align}
Inserting (\ref{eq:err3}) into (\ref{eq:err2}) and using 
\eqref{eq:norms}, (\ref{eq:lengthh}), \eqref{eq:fdo+}, 
as well as (\ref{eq:bddt}) and a corresponding estimate at the right boundary,
we obtain the desired result \eqref{eq:err1a}.
\end{proof}

\begin{lemma} \label{lem:lemma1b} 
Let $0 < \lambda \leq 1$ and $\Delta t \leq \gamma h$. Then
\begin{align}\label{eq:err4}
& h \sum_{i=1}^{5} \sum_{j=1}^{J-1} \vec T^{m,i}_j  \cdot \bigl( \lambda | \delta^1 \vec X^m_j |^2  \frac{\vec E^{m+1}_j- \vec E^m_j}{\Delta t} - \delta^2 \vec E^{m+1}_j \bigr)
+ h \sum_{j=1}^{J-1} \vec T^{m,3}_j \cdot \delta^2 \vec E^{m+1}_j\\ & \quad
\leq  \frac{1}{8 C_0^2} | \vec E^{m+1} |_{2,h}^2  +  C \bigl(  | \vec E^m |_{1,h}^2 + | \vec E^{m+1} |_{1,h}^2 \bigr) + C \, \frac{\gamma}{\Delta t} | \vec E^{m+1} - \vec E^m |_{1,h}^2
\nonumber \\ & + \tfrac18 c_0^2 \lambda \, 
\left|  \frac{\vec E^{m+1}- \vec E^m}{\Delta t} \right|^2_{0,h} 
+C \bigl( h^4 + (\Delta t)^2 \bigr) +  C (\lambda+h) h  \sum_{j=1}^{J-1} \frac{(\delta^1 \vec E^{m+1}_j \cdot \vec e_2)^2}{(\vec X^m_j \cdot \vec e_1)^2}
. \nonumber
\end{align}
\end{lemma}
\begin{proof}
We recall
the definitions of the 
terms $\vec T^{m,i}$ in \eqref{eq:err1}. 
Then we note from
\eqref{eq:esta}, \eqref{eq:lengthh}, (\ref{eq:Xje1}), (\ref{eq:xje1}) and (\ref{eq:maxest}) that
\begin{align*}
| \vec T^{m,1}_j |+ | \vec T^{m,2}_j | & \leq C\bigl( | \delta^2 \vec x^{m+1}_j | + \frac{| \delta^1 \vec X^{m+1}_j \cdot \vec e_2|}{\vec X^m_j \cdot \vec e_1} \bigr) | \delta^1 \vec E^m_j |
\\
& \leq C \bigl( 1+ \frac{| \delta^1 \vec x^{m+1}_j \cdot \vec e_2|}{ \vec x^{m+1}_j \cdot \vec e_1} \frac{\vec x^{m+1}_j \cdot \vec e_1}{\vec X^m_j \cdot \vec e_1} +
\frac{| \delta^1 \vec E^{m+1}_j \cdot \vec e_2|}{\vec X^m_j \cdot \vec e_1} \bigr) | \delta^1 \vec E^m_j | \\
& \leq C | \delta^1 \vec E^m_j| + C h^{\frac{1}{2}} \frac{| \delta^1 \vec E^{m+1}_j \cdot \vec e_2|}{\vec X^m_j \cdot \vec e_1},
\quad 1 \leq j \leq J-1,
\end{align*}
so that \eqref{eq:fdo1}, \eqref{eq:fdo+}, \eqref{eq:lengthh} and
\eqref{eq:norms} imply that
\begin{align}  \label{eq:t1t2}
&  h \,  \sum_{i=1}^2 \sum_{j=1}^{J-1}  \vec T^{m,i}_j \cdot  \bigl( \lambda \, | \delta^1 \vec X^m_j |^2  \frac{\vec E^{m+1}_j- \vec E^m_j}{\Delta t} - \delta^2 \vec E^{m+1}_j \bigr) 
 \\ & \
\leq C h \sum_{j=1}^{J-1} \left( | \delta^- \vec E^m_j | + | \delta^- \vec E^m_{j+1} | + h^{\frac{1}{2}} \frac{| \delta^1 \vec E^{m+1}_j \cdot \vec e_2|}{\vec X^m_j \cdot \vec e_1} \right)
\left( \lambda \left| \frac{\vec E^{m+1}_j- \vec E^m_j}{\Delta t} \right| +
| \delta^2 \vec E^{m+1}_j | \right) \nonumber \\ & \
\leq \epsilon \lambda \left| \frac{\vec E^{m+1} - \vec E^m}{\Delta t} \right|^2_{0,h} +
\epsilon \, | \vec E^{m+1} |^2_{2,h} + C_\epsilon  \, |  \vec E^m |_{1,h}^2 
+ C_\epsilon h^2 \sum_{j=1}^{J-1} \frac{(\delta^1 \vec E^{m+1}_j \cdot \vec e_2)^2}{(\vec X^m_j \cdot \vec e_1)^2}.
\nonumber 
\end{align}
The term involving the product of $\vec T^{m,3}$ with 
$\delta^2 \vec E^{m+1}$ is not estimated, while
\begin{align} \label{eq:errt3}
& h \, \sum_{j=1}^{J-1}  \vec T^{m,3}_j \cdot   \lambda \,  | \delta^1 \vec X^m_j |^2  \frac{\vec E^{m+1}_j- \vec E^m_j}{\Delta t} 
 \\ & \quad
\leq C \lambda  \left( h \, \sum_{j=1}^{J-1} \frac{(\delta^1 \vec E^{m+1}_j \cdot \vec e_2)^2}{(\vec X^m_j \cdot \vec e_1)^2} \right)^{\frac{1}{2}} \, 
\left| \frac{\vec E^{m+1}- \vec E^m}{\Delta t} \right|_{0,h} 
\nonumber \\ & \quad
\leq \epsilon \lambda \left| \frac{\vec E^{m+1}- \vec E^m}{\Delta t} \right|_{0,h}^2 + C_\epsilon \lambda h \,  \sum_{j=1}^{J-1} \frac{(\delta^1 \vec E^{m+1}_j \cdot \vec e_2)^2}{(\vec X^m_j \cdot \vec e_1)^2}.\nonumber
\end{align}

\noindent
Next, we have
\begin{equation}  \label{eq:t4a}
| \vec T^{m,4}_j | \leq C \frac{| \vec E^m_j \cdot \vec e_1 |}{(\vec x^m_j \cdot \vec e_1) ( \vec X^m_j \cdot \vec e_1)} \, 
| \delta^1 \vec x^{m+1}_j \cdot \vec e_2 |, \quad 1 \leq j \leq J-1.
\end{equation}
In addition, (\ref{eq:xje1}) and (\ref{eq:defM}) yield
\[
| \delta^1 \vec x^{m+1}_j \cdot \vec e_2 | = \frac{| \delta^1 \vec x^{m+1}_j \cdot \vec e_2 |}{\vec x^{m+1}_j \cdot \vec e_1}  \vec x^{m+1}_j \cdot \vec e_1 \leq C q_j (1-q_j), \quad 1 \leq j \leq J-1,
\]
so that \eqref{eq:t4a}, (\ref{eq:vje1}), (\ref{eq:xje1}), 
(\ref{eq:Xje1}) and \eqref{eq:inv1} imply
\begin{align*} 
| \vec T^{m,4}_j | & \leq C \, \frac{q_j^2 (1-q_j)^2}{c_2 c_3 q_j^2 (1-q_j)^2} \, \max_{1 \leq k \leq J} | \delta^- \vec E^m_k | \leq C \, 
\max_{1 \leq k \leq J} | \delta^- \vec E^m_k | \\
& \leq  C \bigl( \max_{1 \leq k \leq J} | \delta^- \vec E^{m+1}_k | +  \max_{1 \leq k \leq J} | \delta^-( \vec E^{m+1}_k - \vec E^m_k) | \bigr) \nonumber \\
&  \leq C \max_{1 \leq k \leq J} | \delta^- \vec E^{m+1}_k| + C h^{-\frac{1}{2}} | \vec E^{m+1}- \vec E^m|_{1,h}, \quad 1 \leq j \leq J-1. \nonumber 
\end{align*}

Hence we obtain with the help of (\ref{eq:dsi}) and the fact that $\Delta t \leq \gamma h$
\begin{align} \label{eq:t4}
&  h \, \sum_{j=1}^{J-1} \vec T^{m,4}_j \cdot \bigl( \lambda | \delta^1 \vec X^m_j |^2  \frac{\vec E^{m+1}_j- \vec E^m_j}{\Delta t} - \delta^2 \vec E^{m+1}_j \bigr)
 \\ & \quad
 \leq C \bigl(  \max_{1 \leq j \leq J}  | \delta^- \vec E^{m+1}_j |  + h^{-\frac{1}{2}} | \vec E^{m+1}- \vec E^m|_{1,h} \bigr)
\Bigl( \lambda \left| \frac{ \vec E^{m+1} - \vec E^m}{\Delta t} \right|_{0,h} + | \vec E^{m+1} |_{2,h} \Bigr) \nonumber \\ & \quad
\leq \epsilon | \vec E^{m+1} |^2_{2,h} + \epsilon \lambda  \left| \frac{ \vec E^{m+1} - \vec E^m}{\Delta t} \right|^2_{0,h} 
+ C_\epsilon \, | \vec E^{m+1} |_{1,h}^2 
+ C_\epsilon \, \frac{\gamma}{\Delta t} | \vec E^{m+1} - \vec E^m |_{1,h}^2. \nonumber
\end{align}
Finally, we infer from (\ref{eq:consist}) that
\begin{align}  \label{eq:errt6}
& h \, \sum_{j=1}^{J-1} \vec T^{m,5}_j \cdot \bigl( \lambda | \delta^1 \vec X^m_j |^2  \frac{\vec E^{m+1}_j- \vec E^m_j}{\Delta t} - \delta^2 \vec E^{m+1}_j \bigr) 
\\ & \quad
\leq C \bigl( h^2 + \Delta t \bigr) \Bigl( \lambda \left| \frac{ \vec E^{m+1} - \vec E^m}{\Delta t} \right|_{0,h} + | \vec E^{m+1} |_{2,h} \Bigr) 
\nonumber \\ & \quad
\leq \epsilon \, | \vec E^{m+1} |^2_{2,h}  + \epsilon \lambda  \left| \frac{ \vec E^{m+1} - \vec E^m}{\Delta t} \right|^2_{0,h}  + C_\epsilon \, 
\bigl( h^4 + (\Delta t)^2 \bigr). \nonumber 
\end{align}
If we add \eqref{eq:t1t2}, \eqref{eq:errt3}, \eqref{eq:t4} and \eqref{eq:errt6}, 
and then choose $\epsilon $ sufficiently small, the bound \eqref{eq:err4} follows. 
\end{proof}

In the next lemma we mimic the formal calulations
in \eqref{eq:formal2}, thereby closing our estimates.

\begin{lemma} \label{lem:lemma2} 
There exists a constant $c_4>0$ such that 
\begin{align} \label{eq:t5}
& h \sum_{j=1}^{J-1} \vec T^{m,3}_j \cdot \delta^2 \vec E^{m+1}_j \\ & \quad
\geq c_4 h \, \sum_{j=1}^{J-1} \frac{ ( \delta^1 \vec E^{m+1}_j  \cdot \vec e_2)^2}
{(\vec X^m_j \cdot \vec e_1)^2}
- \tfrac53 h \left( \frac{( \delta^- \vec E^{m+1}_1  \cdot \vec e_2)^2}{(\vec X^m_1 \cdot \vec e_1)^2} + 
\frac{( \delta^+ \vec E^{m+1}_{J-1}  \cdot \vec e_2)^2}{(\vec X^m_{J-1} \cdot \vec e_1)^2} \right) 
\nonumber \\ & \qquad
- \epsilon | \vec E^{m+1} |_{2,h}^2 - C_\epsilon \, | \vec E^{m+1} |_{1,h}^2.
\nonumber
\end{align}
\end{lemma}
\begin{proof}
Let us start by writing
\begin{equation} \label{eq:t5split}
\sum_{j=1}^{J-1} \vec T^{m,3}_j \cdot \delta^2 \vec E^{m+1}_j
= \sum_{j=1}^{J_1} \vec T^{m,3}_j \cdot \delta^2 \vec E^{m+1}_j
+ \sum_{j=J_1+1}^{J-J_1-1} \vec T^{m,3}_j \cdot \delta^2 \vec E^{m+1}_j
+ \sum_{j=J-J_1}^{J-1} \vec T^{m,3}_j \cdot \delta^2 \vec E^{m+1}_j\,,
\end{equation}
and we begin by estimating the first sum on the right hand side of
\eqref{eq:t5split}. 
On recalling \eqref{eq:err1}, we can write
\begin{equation} \label{eq:t5j}
\vec T^{m,3}_j  =  \frac{1}{| \delta^1 \vec x^m_j |} 
\, \frac{\delta^1 \vec E^{m+1}_j \cdot \vec e_2}{\vec X^m_j \cdot \vec e_1}  \, \vec e_2  
   - \frac{1}{| \delta^1 \vec x^m_j |}  \frac{\delta^1 \vec E^{m+1}_j \cdot \vec e_2}{\vec X^m_j \cdot \vec e_1} \Bigl( \frac{(\delta^1 \vec x_j^m)^\perp}{| \delta^1 \vec x^m_j |}
 + \vec e_2 \Bigr)  =: \vec S^1_j+ \vec S^2_j. 
\end{equation}
Observing from \eqref{eq:fdo} that
\begin{align*}
 \bigl( \delta^1 \vec E^{m+1}_j \cdot \vec e_2 \bigr) \delta^2 \vec E^{m+1}_j \cdot \vec e_2  = &
 \frac1h (\delta^1 \vec E^{m+1}_j \cdot \vec e_2) 
 (\delta^+\vec E^{m+1}_j - \delta^- \vec E^{m+1}_j) \cdot \vec e_2 \\
= & \frac1{2 h} \bigl( ( \delta^- \vec E^{m+1}_{j+1}  \cdot \vec e_2)^2 - (\delta^- \vec E^{m+1}_j  \cdot \vec e_2)^2 \bigr),
\end{align*}
we find that 
\begin{align}\label{eq:s1j}
 h \sum_{j=1}^{J_1} \vec S^1_j \cdot  \delta^2  \vec E^{m+1}_j & 
=   \tfrac12 \sum_{j=1}^{J_1} \frac{1}{| \delta^1 \vec x^m_j |} \, \frac{1}{\vec X^m_j \cdot \vec e_1} 
\bigl( ( \delta^- \vec E^{m+1}_{j+1}  \cdot \vec e_2)^2 - (\delta^- \vec E^{m+1}_j  \cdot \vec e_2)^2 \bigr) \\ & 
= \tfrac12 \sum_{j=1}^{J_1}
\frac{(\delta^- \vec E^{m+1}_{j+1}  \cdot \vec e_2)^2}{| \delta^1 \vec x^m_j |\,\vec X^m_j \cdot \vec e_1} 
- \tfrac12 \sum_{j=0}^{J_1-1}
\frac{(\delta^- \vec E^{m+1}_{j+1}  \cdot \vec e_2)^2}{| \delta^1 \vec x^m_{j+1} |\,\vec X^m_{j+1} \cdot \vec e_1}  
 \nonumber \\ & 
= \tfrac12  \sum_{j=1}^{J_1-1} \Bigl( \frac{1}{| \delta^1 \vec x^m_j |} \, \frac{1}{\vec X^m_j \cdot \vec e_1} - \frac{1}{| \delta^1 \vec x^m_{j+1} |} \, \frac{1}{\vec X^m_{j+1} \cdot \vec e_1}
\Bigr) \, ( \delta^- \vec E^{m+1}_{j+1}  \cdot \vec e_2)^2 \nonumber \\ & \quad
- \tfrac12  \frac{( \delta^- \vec E^{m+1}_1  \cdot \vec e_2)^2}{| \delta^1 \vec x^m_1 | \, \vec X^m_1 \cdot \vec e_1}  
+ \tfrac12  \frac{( \delta^- \vec E^{m+1}_{J_1+1}  \cdot \vec e_2)^2}{| \delta^1 \vec x^m_{J_1} |\,\vec X^m_{J_1}
  \cdot \vec e_1}  . \nonumber 
\end{align}
In order to estimate the sum on the right hand side, we observe that
$|\delta^1 \vec x^m_{j+1} | \leq \frac1{2 h} \int_{q_j}^{q_{j+2}} | \vec x^m_\rho | \drho \leq C_0$, recall \eqref{eq:length}, and 
$\vec X^m_j \cdot \vec e_1 \leq \vec X^m_{j+1} \cdot \vec e_1 
\leq C \vec X^m_j \cdot \vec e_1$
for $1\leq j \leq J_1 -1$, recall \eqref{eq:xrhoe1} and \eqref{eq:quotbd}.
Hence we obtain with the help of (\ref{eq:xrhoe1}) that
\begin{align*}
& 
 \frac{1}{| \delta^1 \vec x^m_j |} \, \frac{1}{\vec X^m_j \cdot \vec e_1} -\frac{1}{| \delta^1 \vec x^m_{j+1} |} \, \frac{1}{\vec X^m_{j+1} \cdot \vec e_1}
\\ & \quad
= 
h \, \frac{1}{| \delta^1 \vec x^m_{j+1} |}  \frac{\delta^- \vec X^m_{j+1} \cdot \vec e_1}{ (\vec X^m_{j+1} \cdot \vec e_1)\vec X^m_j \cdot \vec e_1}
-  \frac{1}{\vec X^m_j \cdot \vec e_1} \Bigl( \frac{1}{| \delta^1 \vec x^m_{j+1} |} - \frac{1}{| \delta^1 \vec x^m_j |} \Bigr) \\ & \quad
\geq  h \frac{1}{C_0} \frac{\frac{c_0}{4}}{(\vec X^m_{j+1} \cdot \vec e_1)^2}  - C h \frac{1}{\vec X^m_{j+1} \cdot \vec e_1} \geq  h \frac{c_0}{8 C_0} \frac{1}{(\vec X^m_{j+1} \cdot \vec e_1)^2} - C h, \quad
1\leq j \leq J_1 - 1 .
\end{align*}
Furthermore, on noting 
$| \delta^1 \vec x^m_1 | \geq \delta^1 \vec x^m_1 \cdot \vec e_1 
= \frac1{2 h} \vec x^m_2 \cdot \vec e_1 \geq \frac1{2 h} \vec x^m_1 \cdot \vec e_1 \geq \frac{3}{8}
\frac1h \vec X^m_1 \cdot \vec e_1$, recall (\ref{eq:bdh}), we have that
\[
\tfrac12  \frac{( \delta^- \vec E^{m+1}_1  \cdot \vec e_2)^2}
{| \delta^1 \vec x^m_1 | \,\vec X^m_1 \cdot \vec e_1} 
\leq \tfrac43 h \frac{( \delta^- \vec E^{m+1}_1  \cdot \vec e_2)^2}
{(\vec X^m_1 \cdot \vec e_1)^2}.
\]
Inserting the above two estimates into (\ref{eq:s1j}) yields,
on recalling \eqref{eq:norms}, that 
\begin{align} \label{eq:s1ja}
 h \sum_{j=1}^{J_1} \vec S^1_j \cdot  \delta^2  \vec E^{m+1}_j &
\geq h \frac{c_0}{16 C_0} \sum_{j=1}^{J_1-1} \frac{ ( \delta^- \vec E^{m+1}_{j+1}  \cdot \vec e_2)^2}{(\vec X^m_{j+1} \cdot \vec e_1)^2}
- \tfrac43 h \frac{( \delta^- \vec E^{m+1}_1  \cdot \vec e_2)^2}{(\vec X^m_1 \cdot \vec e_1)^2} 
\\ & \qquad
- C | \vec E^{m+1} |_{1,h}^2.\nonumber 
\end{align} 
Note that in view of \eqref{eq:xtbc}, \eqref{eq:length} and \eqref{eq:leftbd} 
we have $\frac{\vec x_\rho(0,t)}{| \vec x_\rho(0,t)|} =  \vec e_1$,
so that $\frac{\vec x_\rho^\perp(0,t)}{| \vec x_\rho(0,t)|} =
\vec e_1^\perp =-\vec e_2$. Hence \eqref{eq:consist2} and the
smoothness of $\vec x$ imply
\begin{align*}
\left| \frac{(\delta^1 \vec x_j^m)^\perp}{| \delta^1 \vec x^m_j |} + \vec e_2 \right| &
\leq \left| \frac{(\delta^1 \vec x_j^m)^\perp }{| \delta^1 \vec x^m_j |} -
\frac{ \vec x_\rho^\perp(q_j,t_m)}{| \vec x_\rho(q_j,t_m) |} \right| 
+ \left| \frac{ \vec x_\rho^\perp(q_j,t_m)}{| \vec x_\rho(q_j,t_m) |} - \frac{ \vec x_\rho^\perp(0,t_m)}{| \vec x_\rho(0,t_m) |} \right| \nonumber \\ &
\leq C (h + q_j)  \leq C q_j \,,
\end{align*}
for $1\leq j \leq J_1$, 
which means that 
with the help of \eqref{eq:xXjlowerleft} 
we obtain 
\begin{align} \label{eq:s3j}
 h \sum_{j=1}^{J_1} \vec S^2_j \cdot \delta^2  \vec E^{m+1}_j & 
\geq - Ch \sum_{j=1}^{J_1} \bigl( | \delta^- \vec E^{m+1}_j | + | \delta^+ \vec E^{m+1}_j |) \, | \delta^2 \vec E^{m+1}_j |  \\ &
\geq - \epsilon | \vec E^{m+1} |_{2,h}^2 - C_\epsilon | \vec E^{m+1} |_{1,h}^2.
\nonumber
\end{align}
Combining (\ref{eq:s1ja}) and (\ref{eq:s3j}) with (\ref{eq:t5j}) we obtain
\begin{align} \label{eq:t5ja}
 h \sum_{j=1}^{J_1} \vec T^{m,3}_j \cdot \delta^2 \vec E^{m+1}_j 
& \geq h \frac{c_0}{16 C_0} \sum_{j=2}^{J_1} 
\frac{ ( \delta^- \vec E^{m+1}_{j}  \cdot \vec e_2)^2}{(\vec X^m_{j} \cdot \vec e_1)^2}
- \tfrac43 h \frac{( \delta^- \vec E^{m+1}_1  \cdot \vec e_2)^2}{(\vec X^m_1 \cdot \vec e_1)^2} \\ & \quad
- \epsilon | \vec E^{m+1} |_{2,h}^2 - C_\epsilon \, | \vec E^{m+1} |_{1,h}^2.
\nonumber 
\end{align}
In order to estimate the third sum on the right hand side of 
\eqref{eq:t5split}, we start from
\[
\vec T^{m,3}_j  
= - \frac{1}{| \delta^1 \vec x^m_j |} \,
 \frac{\delta^1 \vec E^{m+1}_j  \cdot \vec e_2}{\vec X^m_j \cdot \vec e_1}  \vec e_2 
  - \frac{1}{| \delta^1 \vec x^m_j |}   \frac{\delta^1 \vec E^{m+1}_j \cdot \vec e_2}{\vec X^m_j \cdot \vec e_1}  \Bigl( \frac{(\delta^1 \vec x_j^m)^\perp }{| \delta^1 \vec x^m_j |}
 - \vec e_2 \Bigr) ,
\]
and use similar arguments as above to obtain
\begin{align} \label{eq:t5jb}
 h \sum_{j=J-J_1}^{J-1} \vec T^{m,3}_j \cdot \delta^2 \vec E^{m+1}_j 
& \geq h \frac{c_0}{16 C_0} \sum_{j=J-J_1+1}^{J-1} \frac{ ( \delta^- \vec E^{m+1}_j  \cdot \vec e_2)^2}{(\vec X^m_j \cdot \vec e_1)^2}
- \tfrac43 h \frac{( \delta^- \vec E^{m+1}_J  \cdot \vec e_2)^2}{(\vec X^m_{J-1} \cdot \vec e_1)^2} \\ & \quad
- \epsilon | \vec E^{m+1} |_{2,h}^2 - C_\epsilon \, | \vec E^{m+1} |_{1,h}^2.
\nonumber 
\end{align}
Moreover, it follows from \eqref{eq:length}, 
\eqref{eq:lengthh} and \eqref{eq:xmje1}, that
\begin{align} \label{eq:t5jc}
&  h \sum_{j=J_1+1}^{J-J_1-1} \vec T^{m,3}_j \cdot \delta^2 \vec E^{m+1}_j   - h \frac{c_0}{16 C_0} \sum_{j=J_1+1}^{J-J_1}
\frac{ ( \delta^- \vec E^{m+1}_j  \cdot \vec e_2)^2}{(\vec X^m_j \cdot \vec e_1)^2} \\ & \quad
\geq - C  h \sum_{j=1}^{J-1} \bigl( | \delta^- \vec E^{m+1}_j | + | \delta^- \vec E^{m+1}_{j+1} | \bigr)
| \delta^2 \vec E^{m+1}_j | - C h \sum_{j=1}^{J-1} | \delta^- \vec E^{m+1}_j |^2
\nonumber \\ & \quad
\geq - \epsilon \, | \vec E^{m+1} |_{2,h}^2 -
C_\epsilon | \vec E^{m+1} |_{1,h}^2. \nonumber 
\end{align}
If we combine (\ref{eq:t5ja}), (\ref{eq:t5jb}) and (\ref{eq:t5jc}) we obtain
\begin{align} \label{eq:t5jd}
& h \sum_{j=1}^{J-1} \vec T^{m,3}_j \cdot \delta^2 \vec E^{m+1}_j 
\geq h \frac{c_0}{16 C_0} \sum_{j=2}^{J-1}
\frac{ ( \delta^- \vec E^{m+1}_j  \cdot \vec e_2)^2}{(\vec X^m_j \cdot \vec e_1)^2} \\ & \quad
- \tfrac43 h \left( \frac{( \delta^- \vec E^{m+1}_1  \cdot \vec e_2)^2}{(\vec X^m_1 \cdot \vec e_1)^2} +  \frac{( \delta^+ \vec E^{m+1}_{J-1}  \cdot \vec e_2)^2}{(\vec X^m_{J-1} \cdot \vec e_1)^2} \right)
- \epsilon | \vec E^{m+1} |_{2,h}^2 - C_\epsilon \, | \vec E^{m+1} |_{1,h}^2. 
\nonumber
\end{align}
Observing that in view of \eqref{eq:fdo+} and (\ref{eq:quotbd})
\begin{displaymath}
\frac{ |  \delta^+ \vec E^{m+1}_j \cdot \vec e_2 |}{\vec X^m_j \cdot \vec e_1} 
\leq \frac{8 C_0}{c_3} \frac{ |  \delta^- \vec E^{m+1}_{j+1} \cdot \vec e_2 |}{\vec X^m_{j+1} \cdot \vec e_1}, \quad 1 \leq j \leq J-2,
\end{displaymath}
we have
\begin{align*}
& h \sum_{j=1}^{J-1} \frac{ (  \delta^1 \vec E^{m+1}_j \cdot \vec e_2 )^2}{(\vec X^m_j \cdot \vec e_1)^2} \leq \frac{h}{2} \sum_{j=1}^{J-1} \frac{ (  \delta^+ \vec E^{m+1}_j \cdot \vec e_2 )^2}{(\vec X^m_j \cdot \vec e_1)^2}
  + \frac{h}{2} \sum_{j=1}^{J-1}  \frac{ (  \delta^- \vec E^{m+1}_j \cdot \vec e_2 )^2}{ (\vec X^m_j \cdot \vec e_1)^2}  \\
& \leq \frac{h}{2} \left(  \frac{ ( \delta^- \vec E^{m+1}_1  \cdot \vec e_2)^2}{(\vec X^m_1 \cdot \vec e_1)^2} + \frac{ ( \delta^+ \vec E^{m+1}_{J-1}  \cdot \vec e_2)^2}{(\vec X^m_{J-1} \cdot \vec e_1)^2} \right) +
C h \sum_{j=2}^{J-1} \frac{ ( \delta^- \vec E^{m+1}_j  \cdot \vec e_2)^2}{(\vec X^m_j \cdot \vec e_1)^2}.
\end{align*}
If we insert this bound into (\ref{eq:t5jd}), we deduce \eqref{eq:t5} provided that $c_4$ is small enough. \end{proof}

Combining Lemmas~\ref{lem:lemma1a}, \ref{lem:lemma1b} and \ref{lem:lemma2} we obtain after choosing $\epsilon$, $\gamma$
and $\lambda$ sufficiently small
\begin{align}\label{eq:err8}
& \frac{1+\lambda}{2 \Delta t} \bigl( | \vec E^{m+1} |_{1,h}^2 - | \vec E^m |_{1,h}^2 \bigr) + \frac{1}{16 C_0^2} | \vec E^{m+1} |_{2,h}^2  + \tfrac18 c_0^2 \lambda \, 
\left|  \frac{\vec E^{m+1}- \vec E^m}{\Delta t} \right|^2_{0,h} 
\\ & \qquad 
\leq C \bigl( | \vec E^m |_{1,h}^2 + | \vec E^{m+1} |_{1,h}^2 \bigr)  
 +C \bigl( h^4 + (\Delta t)^2 \bigr). \nonumber 
\end{align}
Furthermore, we have
\begin{align} \label{eq:err9}
\frac{1}{2 \Delta t} \bigl( | \vec E^{m+1} |_{0,h}^2 - | \vec E^m |_{0,h}^2 \bigr) 
& \leq \tfrac12 \left| \frac{\vec E^{m+1} - \vec E^m}{\Delta t} \right|_{0,h} \bigl( | \vec E^{m+1} |_{0,h} + 
| \vec E^m |_{0,h} \bigr)  \\
& \leq  \tfrac1{16} c_0^2 \lambda \, \left| \frac{\vec E^{m+1} - \vec E^m}{\Delta t} \right|^2_{0,h} + C \bigl( | \vec E^{m+1} |_{0,h}^2 + | \vec E^m |_{0,h}^2 \bigr).\nonumber 
\end{align}
On inserting (\ref{eq:err9}) into (\ref{eq:err8}),
divided by $(1+\lambda)$,
we obtain that there exist constants $c_6>0$ and $C_2>0$ such that
\begin{align}\label{eq:err11}
& \frac{1}{\Delta t} \bigl( \| \vec E^{m+1} \|_{1,h}^2 -  \| \vec E^m \|_{1,h}^2 \bigr) + c_6 \left( | \vec E^{m+1} |_{2,h}^2 +  \left| \frac{\vec E^{m+1} - \vec E^m}{\Delta t} \right|^2_{0,h} \right)
  \\  & \qquad 
\leq C_2 \bigl( \| \vec E^{m+1} \|_{1,h}^2 +  \| \vec E^m \|_{1,h}^2 \bigr) 
 +C_2 \bigl( h^4 + (\Delta t)^2 \bigr) . \nonumber
\end{align}
Combining \eqref{eq:err11} with the induction hypothesis \eqref{eq:indass}
completes the proof of Theorem~\ref{thm:main}.
In fact, if we choose $h_0$ so small that 
$C_2 \Delta t \leq \frac12$ for $\Delta t \leq \gamma h_0$, then
$0 < (1 - C_2\,\Delta t)^{-1} \leq 1 + 2 C_2 \Delta t$, and so it follows from
\eqref{eq:err11} and \eqref{eq:indass} that
\begin{align*}
 \| \vec E^{m+1} \|_{1,h}^2 & \leq (1- C_2 \Delta t)^{-1} \left[ \bigl( 1 + C_2 \Delta t \bigr)  \| \vec E^m \|_{1,h}^2  + C_2 \Delta t \bigl( h^4 + (\Delta t)^2 \bigr) \right]
\\ &
\leq \bigl( 1 + 2 C_2 \Delta t \bigr)^2 \| \vec E^m \|_{1,h}^2  + C_2 \bigl( 1 + 2C_2 \Delta t \bigr) \Delta t \bigl( h^4 + (\Delta t)^2 \bigr) 
\\ & 
\leq \bigl( 1 + 2 C_2 \Delta t \bigr)^2 \bigl( h^4 + (\Delta t)^2 \bigr) e^{\mu t_m} + 2C_2\Delta t \bigl( h^4 + (\Delta t)^2 \bigr) 
\\ & 
\leq \bigl( 1 + 3 C_2 \Delta t \bigr)^2 \bigl( h^4 + (\Delta t)^2 \bigr) e^{\mu t_m} \\ &
\leq \bigl( h^4 + (\Delta t)^2 \bigr) e^{6C_2 \Delta t} e^{\mu t_m} = \bigl( h^4 + (\Delta t)^2 \bigr) e^{\mu t_{m+1}},
\end{align*}
if we choose $\mu = 6C_2$.
Since $\mu$, as well as $\gamma$, were chosen independently of $h$ and 
$\ttau$, we have shown \eqref{eq:indass} by induction.
Together with \eqref{eq:dsi0} this proves the inequality \eqref{eq:main}. 
Finally, multiplying \eqref{eq:err11} by $\Delta t$
and summing for $m=0,\ldots,M-1$ yields the bound \eqref{eq:maindt}. 

\setcounter{equation}{0}
\section{Numerical results} \label{sec:nr}

It is easy to show that a shrinking sphere with 
radius $[1 - 4\,t]^\frac12$ is a solution to (\ref{eq:mcfS}). 
In fact, the parameterization 
\begin{equation} \label{eq:solxpi}
\vec x(\rho, t) = 
[1 - 4\,t]^\frac12\, \begin{pmatrix}
\sin (\pi\,\rho) \\ \cos (\pi\,\rho)
\end{pmatrix}
\end{equation}
solves \eqref{eq:xtnew}. 
On letting $\vec x^m_j = \vec x(q_j,t_m)$, $j=0,\ldots,J$,
we compare \eqref{eq:solxpi} 
to the discrete solutions $(\vec X^{m})_{m=0,\ldots,M}$ of \eqref{eq:fd} 
and perform two convergence experiments. 
In particular, we choose 
either $\ttau = h$ or $\ttau = h^2$, 
for $h = J^{-1} = 2^{-k}$, $k=5,\ldots,9$. 

The results in Tables~\ref{tab:spherefd3} and \ref{tab:spherefd2} 
confirm the theoretical results proved in Theorem~\ref{thm:main}.
We stress that the quadratic convergence rate for the $H^1$--seminorm 
in Table~\ref{tab:spherefd2} is better than the linear 
rate observed in \cite[Table~4]{schemeD} for the finite element scheme
considered there.
This suggests that the delicate treatment of the
boundary nodes in our finite difference scheme \eqref{eq:fd} is crucial to
obtain the optimal convergence rate in Theorem~\ref{thm:main}.
\begin{table}
\center
\caption{Errors for the convergence test for (\ref{eq:solxpi})
over the time interval $[0,0.125]$ with $\ttau = h$.}
\begin{tabular}{|r|c|c|c|c|c|}
\hline
$J$ & $\displaystyle\max_{m=0,\ldots,M} |\vec x^m - \vec X^m|_{0,h}$ & EOC &
$\displaystyle\max_{m=0,\ldots,M} |\vec x^m - \vec X^m|_{1,h}$ & EOC
 \\ \hline
32  & 3.5744e-02 & ---  & 1.1225e-01 & ---  \\
64  & 2.0034e-02 & 0.84 & 6.2934e-02 & 0.83 \\
128 & 1.0690e-02 & 0.91 & 3.3582e-02 & 0.91 \\
256 & 5.5352e-03 & 0.95 & 1.7389e-02 & 0.95 \\
512 & 2.8185e-03 & 0.97 & 8.8546e-03 & 0.97 \\
\hline
\end{tabular}
\label{tab:spherefd3}
\end{table}%
\begin{table}
\center
\caption{Errors for the convergence test for (\ref{eq:solxpi})
over the time interval $[0,0.125]$ with $\ttau = h^2$.}
\begin{tabular}{|r|c|c|c|c|c|}
\hline
$J$ & $\displaystyle\max_{m=0,\ldots,M} |\vec x^m - \vec X^m|_{0,h}$ & EOC &
$\displaystyle\max_{m=0,\ldots,M} |\vec x^m - \vec X^m|_{1,h}$ & EOC
 \\ \hline
32  & 1.0024e-03 & ---  & 3.1480e-03 & ---  \\
64  & 2.5201e-04 & 1.99 & 7.9165e-04 & 1.99 \\
128 & 6.3093e-05 & 2.00 & 1.9821e-04 & 2.00 \\
256 & 1.5779e-05 & 2.00 & 4.9571e-05 & 2.00 \\
512 & 3.9451e-06 & 2.00 & 1.2394e-05 & 2.00 \\
\hline
\end{tabular}
\label{tab:spherefd2}
\end{table}%

In Figure~\ref{fig:limacon} we show a simulation for mean curvature flow of a
sphere with an inscribed torus. In particular, the initial surface
selfintersects on the equator of the sphere, and has genus 0.
For the scheme \eqref{eq:fd} we choose $J=1024$ and $\ttau = 10^{-4}$.
Under mean
curvature flow, the torus attempts to shrink to a circle. For the generating
curve, this means that the cusp or swallow tail tries to disappear. Of course,
for the approximated partial differential equation this represents a
singularity, where the curvatures of the curve, and of the corresponding
axisymmetric surface, blow up. However, the discrete scheme \eqref{eq:fd} is
blind to the self-intersection and the associated singularity. Hence the finite
difference approximation simply integrates across the singularity.
The same behavior can be seen, for example, in
\cite[Figure~4.2]{DeckelnickDE05} and \cite[Figure~6]{triplejMC}.
Continuing the evolution in Figure~\ref{fig:limacon}
would show the curve approaching a shrinking semicircle, that
eventually vanishes at the origin.
\begin{figure}
\center
\mbox{
\includegraphics[angle=-90,width=0.16\textwidth]{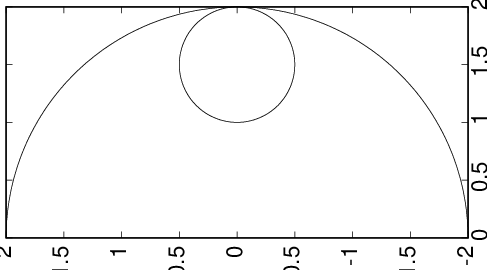}
\includegraphics[angle=-90,width=0.16\textwidth]{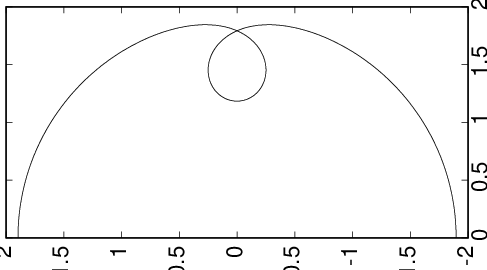}
\includegraphics[angle=-90,width=0.16\textwidth]{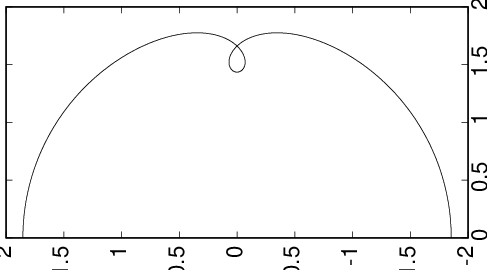}
\includegraphics[angle=-90,width=0.16\textwidth]{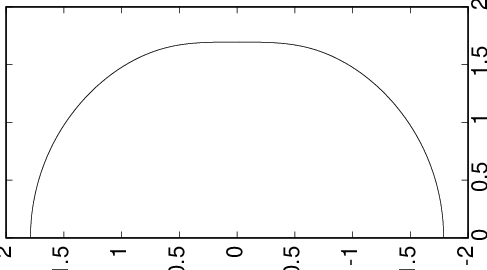}
}
\caption{Evolution for a torus inscribed within a sphere. Plots are at times
$t=0,0.1,0.14,0.2$.
}
\label{fig:limacon}
\end{figure}%

In the recent article \cite{schemeD}, the present authors numerically studied
the Angenent torus, see \cite{Angenent92,Mantegazza11}, as an example of a
self-shrinker for mean curvature flow. Here we recall that the surface
$\mathcal{S}(0)$ is called a self-shrinker, if the self-similar family of 
surfaces
\begin{equation*} 
\mathcal{S}(t) = [1- t]^\frac12 \mathcal{S}(0)
\end{equation*}
is a solution to \eqref{eq:mcfS}. 
In what follows, we would like to use our approximation \eqref{eq:fd} in order
to investigate self-shrinkers of genus-0. It was shown in \cite{KleeneM14} 
that the only bounded embedded genus-0 self-shrinker in $\bR^3$ is the sphere
of radius $2$. Note that the unit sphere has an extinction time of
$T_0 = \frac14$, recall \eqref{eq:solxpi}. On the other hand, in
\cite{DruganK17} the existence of infinitely many immersed self-shrinkers with
rotational symmetry was proved. Hence, inspired by 
\cite[Figure~3]{DruganK17}, we would like compute such a self-similar
evolution for mean curvature flow. To this end, we use the open curve analogue
of \cite[(5.7),(5.8)]{schemeD} in order to calculate a profile curve of a
self-shrinker that has three self-intersections. Using the obtained curve as
initial data for the scheme \eqref{eq:fd} yields the self-similar evolution
displayed in Figure~\ref{fig:DK17b}. Here we used the discretization
parameters $J=512$ and $\ttau = 10^{-4}$. Note that the numerical method
appears to confirm the unit extinction time. In fact, continuing the evolution
until the methods breaks down yields the behaviour of the approximate 
surface area
\[
A^m = 2\pi h \sum_{j=1}^J \vec X^m_j \cdot \vec e_1 |\delta^- \vec X^m_j|
\]
as shown in Figure~\ref{fig:DK17b_sa}, with the expected linear decay and an 
approximate extinction time of 1.
\begin{figure}
\center
\mbox{
\includegraphics[angle=-90,width=0.25\textwidth]{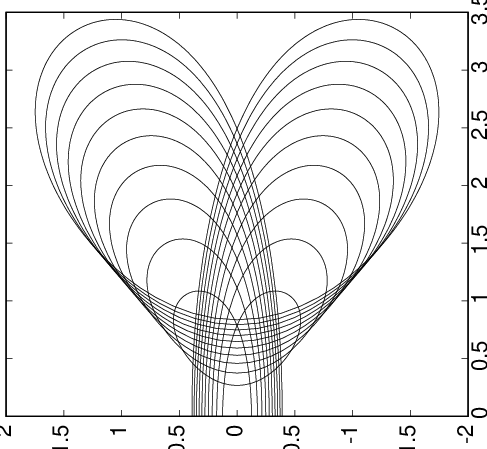}
\includegraphics[angle=-90,width=0.25\textwidth]{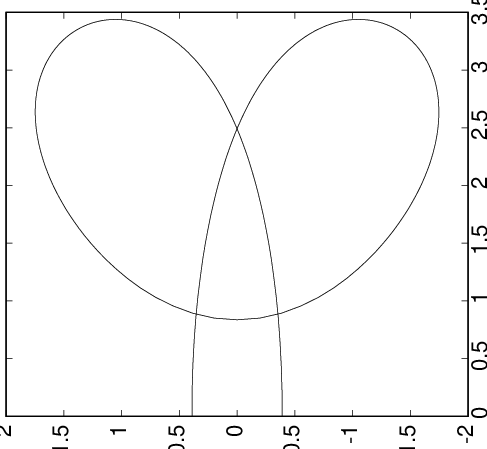}
\includegraphics[angle=-90,width=0.25\textwidth]{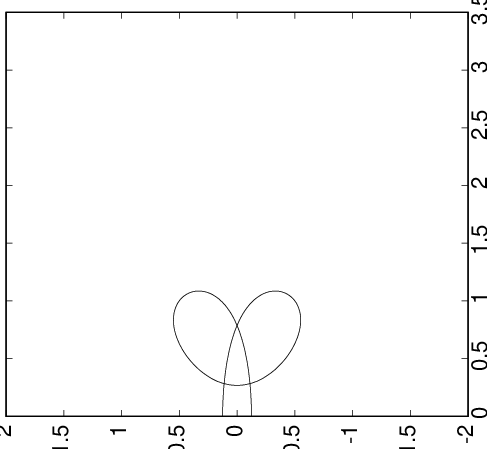}
}
\caption{Self-similar evolution for a surface with three self-intersections.
Plots are at times $t=0,0.1,\ldots,0.9$, and again at times $t=0$ and $t=0.9$.
}
\label{fig:DK17b}
\end{figure}%
\begin{figure}
\center
\includegraphics[angle=-90,width=0.4\textwidth]{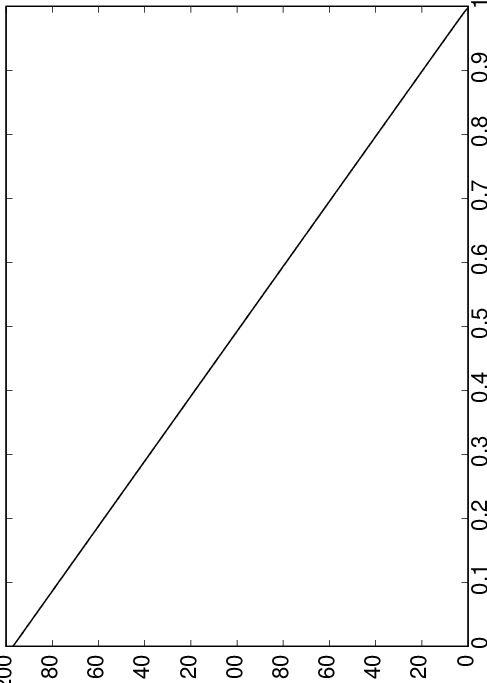}
\caption{A plot of the approximate surface area $A^m$,
for the simulation in Figure~\ref{fig:DK17b}, over time.
}
\label{fig:DK17b_sa}
\end{figure}%

Finally, we include a numerical experiment to demonstrate that our scheme can
also deal with initial data that violate the $90^\circ$ contact angle condition 
in \eqref{eq:xtbc}. To this end, in Figure~\ref{fig:cusps} we start a simulation
for a surface that has two cone singularities: an inward cone and an outward 
cone. The
generating curve has a $45^\circ$ contact angle at the axis of rotation, which
induces a discontinuous jump in time for the solution of the partial
differential equation. 
For the simulation 
we choose $J=512$ and $\ttau = 10^{-4}$ for the scheme \eqref{eq:fd}. 
It can be observed that the outward cone very quickly
smoothens to a rounded tip, while the inward cone also smoothens and rises
at the same time. Eventually the curve approaches a shrinking semicircle, that
will shrink to a point.
\begin{figure}
\center
\mbox{
\includegraphics[angle=-90,width=0.16\textwidth]{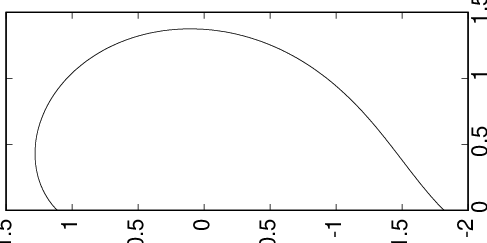}\
\includegraphics[angle=-90,width=0.16\textwidth]{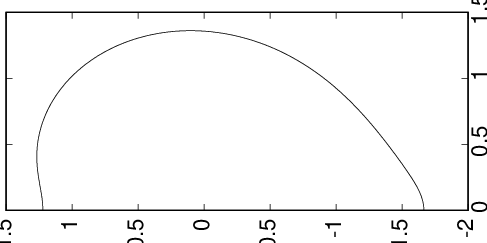}\
\includegraphics[angle=-90,width=0.16\textwidth]{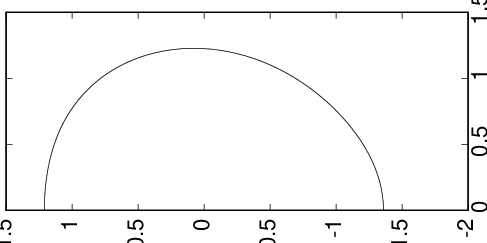}\
\includegraphics[angle=-90,width=0.16\textwidth]{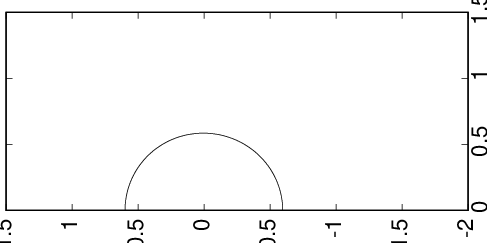}
}
\caption{Evolution for a surface with two cone singularities. Plots are at times
$t=0,0.01,0.1,0.4$.
}
\label{fig:cusps}
\end{figure}%

\newpage
\begin{appendix}
\renewcommand{\theequation}{\Alph{section}.\arabic{equation}}

\setcounter{equation}{0} 
\section{Properties of the solution} \label{sec:A}

\begin{lemma}[Behaviour at the boundary] 
Let $\vec x: [0,1] \times [0,T] \to \bR^2$ 
satisfy {\rm Assumption~\ref{ass:smooth}}. 
Then we have 
\begin{subequations} \label{eq:bc123}
\begin{alignat}{2}
\vec x_{\rho \rho}\cdot \vec e_1 =
\vec x_{\rho \rho}\cdot \vec x_\rho & =0
\quad &&\text{on } \{0,1\}\times [0,T],\label{eq:bc2} \\
\vec x_{\rho \rho \rho} \cdot \vec e_2 & =0
\quad &&\text{on } \{0,1\}\times [0,T],\label{eq:bc3} \\
\vec x_t & = 
2 \frac{\vec x_{\rho \rho} \cdot \vec e_2}{| \vec x_\rho |^2} 
\vec e_2 \quad &&\text{on } \{0,1\}\times [0,T]. \label{eq:speedbd} 
\end{alignat}
\end{subequations}
\end{lemma}
\begin{proof} 
We have from \eqref{eq:xtbc} that
\begin{equation} \label{eq:xrhoperp}
\vec x_\rho^\perp(0,t) = (| \vec x_\rho(0,t) | \, \vec e_1)^\perp 
= - | \vec x_\rho(0,t) | \, \vec e_2, 
\end{equation}
and so we obtain with the help of  L'Hospital's rule that
\begin{align} \label{eq:hosp}
\lim_{\rho \searrow 0} \frac{1}{| \vec x_\rho(\rho,t) |^2} \frac{ \vec x_\rho(\rho,t) \cdot \vec e_2}{\vec x(\rho,t) \cdot \vec e_1} \vec x_\rho^\perp(\rho,t)
& = 
\frac{1}{| \vec x_\rho(0,t) |^2} \frac{ \vec x_{\rho \rho} (0,t) \cdot \vec e_2}{\vec x_\rho(0,t) \cdot \vec e_1} \vec x_\rho^\perp(0,t) \\ & 
= -  \frac{\vec x_{\rho \rho} (0,t) \cdot \vec e_2}{| \vec x_\rho(0,t) |^2} \vec e_2.\nonumber 
\end{align}
Thus (\ref{eq:xtnewflow}) implies that
\begin{equation}  \label{eq:xtbd}
\vec x_t(0,t) = \frac{\vec x_{\rho \rho}(0,t)}{| \vec x_\rho(0,t) |^2} + \frac{\vec x_{\rho \rho} (0,t) \cdot \vec e_2}{| \vec x_\rho(0,t) |^2} \vec e_2.
\end{equation}
Observing from \eqref{eq:xtbc} that 
$\vec x_t(0,t) \cdot \vec e_1 =0$, we infer from \eqref{eq:xtbd} that 
$\vec x_{\rho \rho}(0,t) \cdot \vec e_1 =0$,
which together with \eqref{eq:xtbc} proves \eqref{eq:bc2} at $\rho=0$. 
In particular, 
$\vec x_{\rho \rho}(0,t) = (\vec x_{\rho \rho}(0,t) \cdot \vec e_2) \vec e_2$.
Combining this with (\ref{eq:xtbd}) yields (\ref{eq:speedbd}) at $\rho=0$.
In order to prove (\ref{eq:bc3}),
we differentiate (\ref{eq:xtnewflow}) with respect to $\rho$ and obtain
\begin{equation}
\vec x_{t \rho} = \frac{\vec x_{\rho \rho \rho}}{| \vec x_\rho |^2} - 2 \frac{\vec x_{\rho \rho} \cdot \vec x_\rho}{| \vec x_\rho |^4} \vec x_{\rho \rho} + 
\frac{\vec x_\rho \cdot \vec e_2}{\vec x \cdot \vec e_1} \bigl( 2 \frac{\vec x_{\rho \rho} \cdot \vec x_\rho}{| \vec x_\rho |^4} \vec x_\rho^\perp - \frac{1}{| \vec x_\rho |^2} \vec x_{\rho \rho}^\perp \bigr) 
 - \bigl( \frac{\vec x_\rho \cdot \vec e_2}{\vec x \cdot \vec e_1} \bigr)_\rho \, \frac{1}{| \vec x_\rho |^2} \vec x_\rho^\perp 
\label{eq:xtrho}
\end{equation}
in $(0,1) \times (0,T]$.
A further application of L'Hospital's rule implies that
\begin{align}\label{eq:limitbd}
\lim_{\rho \searrow 0} \bigl( \frac{\vec x_\rho \cdot \vec e_2}{\vec x \cdot \vec e_1} \bigr)_\rho(\rho,t) &= \lim_{\rho \searrow 0} \frac{ ( \vec x_{\rho \rho}(\rho,t) \cdot \vec e_2) ( \vec x(\rho,t) \cdot \vec e_1) -
(\vec x_\rho(\rho,t) \cdot \vec e_2) ( \vec x_\rho(\rho,t) \cdot \vec e_1)}{(\vec x(\rho,t) \cdot \vec e_1)^2} \\
&= \lim_{\rho \searrow 0} \frac{ ( \vec x_{\rho \rho \rho}(\rho,t) \cdot \vec e_2)( \vec x(\rho,t) \cdot \vec e_1) - ( \vec x_\rho(\rho,t) \cdot \vec e_2)( \vec x_{\rho \rho}(\rho,t) \cdot \vec e_1)}{
2 (\vec x(\rho,t) \cdot \vec e_1) (\vec x_{\rho}(\rho,t) \cdot \vec e_1)} \nonumber \\
& = \tfrac12 \frac{\vec x_{\rho \rho \rho}(0,t) \cdot  \vec e_2}{ \vec x_\rho(0,t) \cdot \vec e_1}, \nonumber 
\end{align}
since $\vec x_{\rho \rho}(0,t) \cdot \vec e_1=0$. 
Combining (\ref{eq:xtrho}) and (\ref{eq:limitbd}), on noting
\eqref{eq:bc2}, 
\eqref{eq:xtbc} and \eqref{eq:xrhoperp}, yields that 
\begin{align} \label{eq:xtrho0}
\vec x_{t \rho}(0,t) = 
\frac{\vec x_{\rho \rho \rho}(0,t)}{| \vec x_\rho(0,t) |^2} 
+ \tfrac12 \frac{\vec x_{\rho \rho \rho}(0,t) \cdot  \vec e_2}
{| \vec x_\rho(0,t) |^2} \vec e_2.
\end{align}
Since $\vec x_{t \rho}(0,t) \cdot \vec e_2 =0$ in view of (\ref{eq:xtbc}), we 
deduce from (\ref{eq:xtrho0}) that also (\ref{eq:bc3}) holds at the left 
boundary point.
The proof of \eqref{eq:bc123} for the other boundary point is analogous.
\end{proof}

\begin{lemma} 
Let $\vec x: [0,1] \times [0,T] \to \bR^2$ 
satisfy {\rm Assumption~\ref{ass:smooth}}. 
Then there exists $0<c_2 < \tilde c_2$ such that
\begin{align}
c_2 \rho (1 - \rho) \leq \vec x(\rho,t) \cdot \vec e_1 \leq  \tilde c_2 \rho (1 - \rho) 
\quad \mbox{ for all } (\rho,t) \in [0,1] \times [0,T]. \label{eq:xje1}
\end{align}
Moreover, 
there exists $K >0$ such that for all $0 < h  \leq \tfrac12\delta$,
with $\delta$ as in \eqref{eq:leftright},
\begin{align} 
\frac{1}{\vec x \cdot \vec e_1} \left| 
\frac{ \partial^{\ell}_t \vec x(\cdot+h,t) 
- \partial^{\ell}_t \vec x(\cdot-h,t) }{2h} \cdot \vec e_2
\right| \leq K 
\quad \mbox{ in } [h, 1-h] \times [0,T]
,\ \ell = 0,1. \label{eq:defM}
\end{align}
\end{lemma}
\begin{proof} 
The result \eqref{eq:xje1} is an immediate consequence of (\ref{eq:leftright}).

Let $t \in [0,T]$ and $h \leq \rho \leq \frac12 \delta$. 
We infer from (\ref{eq:xtbc}) and (\ref{eq:leftbd}) that $\vec x(\rho,t) \cdot \vec e_1 \geq \frac12 c_0 \rho$ and hence
\begin{align*} 
 \frac1{\vec x(\rho,t) \cdot \vec e_1} \left| 
\frac{  \vec x(\rho+h,t) - \vec x(\rho-h,t) }{2h} \cdot \vec e_2 \right|
&\leq \frac2{c_0 \rho}  \left| \frac1{2h} \int_{\rho -h}^{\rho +h} 
\bigl( \vec x_\rho(\zeta,t) - \vec x_\rho(0,t) \bigr) \cdot \vec e_2 \dzeta \right| \\ & 
\leq \frac2{c_0}  \max_{[0,\delta]} | \vec x_{\rho \rho}(\cdot,t) \cdot \vec e_2 | \\ &
\leq \frac2{c_0}  \max_{[0,\delta] \times [0,T]} | \vec x_{\rho \rho}\cdot \vec e_2 | . 
\end{align*}
We can argue in the same way for $1-\frac12 \delta \leq \rho \leq 1 -h$, 
while for $\frac12\delta \leq \rho \leq 1-\frac12 \delta$ we have that 
\begin{align*}
\frac1{\vec x(\rho,t) \cdot \vec e_1} \left| 
\frac{ \vec x(\rho+h,t) - \vec x(\rho-h,t) }{2h} \cdot \vec e_2\right| 
& \leq \frac1{c_1} \left| \frac1{2h} \int_{\rho-h}^{\rho+h} 
 \vec x_\rho(\zeta,t)  \cdot \vec e_2 \dzeta \right| \\ &
\leq \frac1{c_1} 
\max_{[\frac{\delta}{2},1-\frac{\delta}{2}]\times[0,T]} 
| \vec x_{\rho} \cdot \vec e_2 |,
\end{align*}
so that (\ref{eq:defM}) holds with $K=\max \{ \frac2{c_0} 
\max_{[0,\delta]\times[0,T]} | \vec x_{\rho \rho} \cdot \vec e_2 |, 
\frac1{c_1}  
\max_{[\frac{\delta}{2},1-\frac{\delta}{2}]\times[0,T]} | \vec x_{\rho} \cdot \vec e_2 | \}$
in the case $\ell=0$. 
The case $\ell=1$ can be treated in the same way, on noting
that $\vec x_{t \rho}(q,t) \cdot \vec e_2 =0$ for $q \in \lbrace 0,1 \rbrace$.
\end{proof}

\end{appendix}

\def\soft#1{\leavevmode\setbox0=\hbox{h}\dimen7=\ht0\advance \dimen7
  by-1ex\relax\if t#1\relax\rlap{\raise.6\dimen7
  \hbox{\kern.3ex\char'47}}#1\relax\else\if T#1\relax
  \rlap{\raise.5\dimen7\hbox{\kern1.3ex\char'47}}#1\relax \else\if
  d#1\relax\rlap{\raise.5\dimen7\hbox{\kern.9ex \char'47}}#1\relax\else\if
  D#1\relax\rlap{\raise.5\dimen7 \hbox{\kern1.4ex\char'47}}#1\relax\else\if
  l#1\relax \rlap{\raise.5\dimen7\hbox{\kern.4ex\char'47}}#1\relax \else\if
  L#1\relax\rlap{\raise.5\dimen7\hbox{\kern.7ex
  \char'47}}#1\relax\else\message{accent \string\soft \space #1 not
  defined!}#1\relax\fi\fi\fi\fi\fi\fi}


\begin{thebibliography}{10}

\bibitem{Angenent92}
{\sc S.~B. Angenent}, {\em Shrinking doughnuts}, in Nonlinear diffusion
  equations and their equilibrium states, 3 ({G}regynog, 1989), vol.~7 of
  Progr. Nonlinear Differential Equations Appl., Birkh\"{a}user Boston, Boston,
  MA, 1992, pp.~21--38.

\bibitem{schemeD}
{\sc J.~W. Barrett, K.~Deckelnick, and R.~N\"urnberg}, {\em A finite element
  error analysis for axisymmetric mean curvature flow}, IMA J. Numer. Anal.,
  (2020).
\newblock (to appear).

\bibitem{triplejMC}
{\sc J.~W. Barrett, H.~Garcke, and R.~N\"urnberg}, {\em On the variational
  approximation of combined second and fourth order geometric evolution
  equations}, SIAM J. Sci. Comput., 29 (2007), pp.~1006--1041.

\bibitem{gflows3d}
\leavevmode\vrule height 2pt depth -1.6pt width 23pt, {\em On the parametric
  finite element approximation of evolving hypersurfaces in {${\mathbb R}^3$}},
  J. Comput. Phys., 227 (2008), pp.~4281--4307.

\bibitem{aximcf}
\leavevmode\vrule height 2pt depth -1.6pt width 23pt, {\em Variational
  discretization of axisymmetric curvature flows}, Numer. Math., 141 (2019),
  pp.~791--837.

\bibitem{bgnreview}
\leavevmode\vrule height 2pt depth -1.6pt width 23pt, {\em Parametric finite
  element approximations of curvature driven interface evolutions}, in Handb.
  Numer. Anal., A.~Bonito and R.~H. Nochetto, eds., vol.~21, Elsevier,
  Amsterdam, 2020, pp.~275--423.

\bibitem{DeckelnickD95}
{\sc K.~Deckelnick and G.~Dziuk}, {\em On the approximation of the curve
  shortening flow}, in Calculus of Variations, Applications and Computations
  (Pont-\`a-Mousson, 1994), C.~Bandle, J.~Bemelmans, M.~Chipot, J.~S.~J.
  Paulin, and I.~Shafrir, eds., vol.~326 of Pitman Res. Notes Math. Ser.,
  Longman Sci. Tech., Harlow, 1995, pp.~100--108.

\bibitem{DeckelnickDE05}
{\sc K.~Deckelnick, G.~Dziuk, and C.~M. Elliott}, {\em Computation of geometric
  partial differential equations and mean curvature flow}, Acta Numer., 14
  (2005), pp.~139--232.

\bibitem{DruganK17}
{\sc G.~Drugan and S.~J. Kleene}, {\em Immersed self-shrinkers}, Trans. Amer.
  Math. Soc., 369 (2017), pp.~7213--7250.

\bibitem{Dziuk91}
{\sc G.~Dziuk}, {\em An algorithm for evolutionary surfaces}, Numer. Math., 58
  (1991), pp.~603--611.

\bibitem{Dziuk94}
\leavevmode\vrule height 2pt depth -1.6pt width 23pt, {\em Convergence of a
  semi-discrete scheme for the curve shortening flow}, Math. Models Methods
  Appl. Sci., 4 (1994), pp.~589--606.

\bibitem{Ecker04}
{\sc K.~Ecker}, {\em Regularity Theory for Mean Curvature Flow}, Birkh\"auser,
  Boston, 2004.

\bibitem{ElliottF17}
{\sc C.~M. Elliott and H.~Fritz}, {\em On approximations of the curve
  shortening flow and of the mean curvature flow based on the {D}e{T}urck
  trick}, IMA J. Numer. Anal., 37 (2017), pp.~543--603.

\bibitem{GarckeM20}
{\sc H.~Garcke and B.-V. Matioc}, {\em On a degenerate parabolic system
  describing the mean curvature flow of rotationally symmetric closed
  surfaces}, J. Evol. Equ.,  (2020).
\newblock (to appear).

\bibitem{KleeneM14}
{\sc S.~Kleene and N.~M. M{\o}ller}, {\em Self-shrinkers with a rotational
  symmetry}, Trans. Amer. Math. Soc., 366 (2014), pp.~3943--3963.

\bibitem{KovacsLL19}
{\sc B.~Kov\'{a}cs, B.~Li, and C.~Lubich}, {\em A convergent evolving finite
  element algorithm for mean curvature flow of closed surfaces}, Numer. Math.,
  143 (2019), pp.~797--853.

\bibitem{Mantegazza11}
{\sc C.~Mantegazza}, {\em Lecture notes on mean curvature flow}, vol.~290 of
  Progress in Mathematics, Birkh\"auser/Springer Basel AG, Basel, 2011.

\bibitem{Mierswa20}
{\sc A.~Mierswa}, {\em Error estimates for a finite difference approximation of
  mean curvature flow for surfaces of torus type}, PhD thesis, University
  Magdeburg, Magdeburg, 2020.

\end{thebibliography}
\end{document}